\documentclass[twoside,12pt]{article}

\usepackage{jmlr2e}
\usepackage{tikz}
\usetikzlibrary{arrows,decorations.pathmorphing,backgrounds,positioning,fit,petri,graphs}
\usepackage[textfont=scriptsize,labelfont=scriptsize]{caption}
\usepackage[textfont=scriptsize,labelfont=scriptsize]{subcaption}
\usepackage{multirow}
\usepackage{booktabs}       
\usepackage[linesnumbered,lined,boxed,commentsnumbered,ruled]{algorithm2e}
\usepackage{enumerate}
\input{dlfh.sty} 
\usepackage{dlfhMicros}
\usepackage{amsmath}
\usepackage{color}
\definecolor{gmcolor}{rgb}{.7, 0, .5}

\usepackage{array}
\newcolumntype{P}[1]{>{\centering\arraybackslash}p{#1}}

\firstpageno{1}

\begin{document}

\title{Distributed Learning via Filtered Hyperinterpolation on Manifolds}

\author{\name Guido Mont\'{u}far \email \texttt{montufar@math.ucla.edu} \\
       \addr Department of Mathematics and Department of Statistics\\
       University of California, Los Angeles, USA \\
       Max Planck Institute for Mathematics in the Sciences, Leipzig, Germany
       \AND
       \name Yu Guang Wang \email \texttt{yuguang.wang@mis.mpg.de} \\
       \addr Max Planck Institute for Mathematics in the Sciences, Leipzig, Germany\\
       School of Mathematics and Statistics\\
       University of New South Wales, Sydney, Australia}

\editor{\today}

\maketitle

\begin{abstract}
Learning mappings of data on manifolds is an important topic in contemporary machine learning, with applications in astrophysics, geophysics, statistical physics, medical diagnosis, biochemistry, 3D object analysis. This paper studies the problem of learning real-valued functions on manifolds through filtered hyperinterpolation of input-output data pairs where the inputs may be sampled deterministically or at random and the outputs may be clean or noisy. Motivated by the problem of handling large data sets, it presents a parallel data processing approach which distributes the data-fitting task among multiple servers and synthesizes the fitted sub-models into a global estimator. We prove quantitative relations between the approximation quality of the learned function over the entire manifold, the type of target function, the number of servers, and the number and type of available samples. We obtain the approximation rates of convergence for distributed and non-distributed approaches. For the non-distributed case, the approximation order is optimal.
\end{abstract}


\tableofcontents

\section{Introduction}\label{sec:intro}
Learning functions over manifolds has become an increasingly important topic in machine learning. 
The performance of many machine learning algorithms depends strongly on the geometry of the data. 
In real-world applications, one often has huge data sets with noisy samples.
In this paper, we propose \emph{distributed filtered hyperinterpolation on manifolds}, which combines filtered hyperinterpolation and distributed learning \citep{LiGuZh2017,LiZh2018}.
Filtered hyperinterpolation \citep{SlWo2012,WaLeSlWo2017} provides a constructive approach to modelling mappings between inputs and outputs in a way that can reduce the influence of noise.
The distributed strategy assigns the learning task of the input-output mapping to multiple local servers, enabling parallel computing for massive data sets. 
Each server handles a small fraction of all data by filtered hyperinterpolation. 
It then synthesizes the local estimators as a global estimator. 
We show the precise quantitative relation between the approximation error of the distributed filtered hyperinterpolation, the number of the local servers, and the amount of data. The approximation error (over the entire manifold) converges to zero provided the available amount of data increases sufficiently fast with the number of servers. 

\begin{figure}[tb]
\centering
\begin{tikzpicture}[x=2cm,y=2cm]
\node at (.5,1.75) [fill=black,circle,inner sep=1pt, node distance = .5cm, label=$f^\ast$] (f) {};

\draw[] 
(0,1) to [bend left] 
node[pos=.5, inner sep=0pt,minimum size=0pt] (c) {} 
(.25,0) 
to [bend left] 
node[pos=1.15,above] (n2) {$\Pi_{2n}$} 
(2.25,0) 
to [bend left] 
node[pos=.5, inner sep=0pt,minimum size=0pt] (cc) {} (2,1) 
to [bend right] (0,1)
-- cycle; 

\draw[thick] (c) to [bend left] 
node[pos=.3,fill=black,circle,inner sep=1pt, label={[below right]:$f^\ast_{\Pi_n}$}] (fs) {} 
node[pos=1.15, above] (n) {$\Pi_n$} (cc); 

\node [above right of = fs, fill=black,circle,inner sep=1pt, node distance = .75cm, label={[right]:$V_{D,n}$}] (fts) {};
\path[draw, dotted] (f) to (fs);
\path[draw, dotted] (f) to (fts);
\end{tikzpicture}
\quad 
\begin{tikzpicture}[x=2cm,y=2cm]
\node at (.5,1.75) [fill=black,circle,inner sep=1pt, node distance = .5cm, label=$f^\ast$] (f) {};
\draw[] 
(0,1) to [bend left] 
node[pos=.5, inner sep=0pt,minimum size=0pt] (c) {} 
(.25,0) 
to [bend left] 
node[pos=1.15, above] (n2) {$\Pi_{2n}$} 
(2.25,0) 
to [bend left] 
node[pos=.5, inner sep=0pt,minimum size=0pt] (cc) {} (2,1) 
to [bend right] (0,1)
-- cycle; 

\draw[thick] (c) to [bend left] 
node[pos=.3,fill=black,circle,inner sep=1pt, label={[below right]:$f^\ast_{\Pi_n}$}] (fs) {} 
node[pos=1.15, above] (n) {$\Pi_n$} (cc); 

\node [above right of = fs, fill=black,circle,inner sep=1pt, node distance = .5cm, label={[right]:\small$V_{D,n}^{(m)}$}
] (fts) {};

\node [above right of = fts, fill=black,circle,inner sep=.75pt, node distance = .7cm, label={[right]:\tiny$V_{D_1,n}$}
] (fts1) {};

\node [left of = fts, fill=black,circle,inner sep=.75pt, node distance = .75cm, label={[left]:\tiny$V_{D_2,n}$}
] (fts2) {};

\node [below right of = fts, fill=black,circle,inner sep=.75pt, node distance = .75cm, label={[below]:\tiny$\phantom{asfasf} V_{D_3,n}$}
] (fts3) {};
\path[draw, dotted] (f) to (fs);
\path[draw, dotted] (f) to (fts);
\end{tikzpicture}
\caption{Illustration of approximations computed on a single server and distributed servers. 
In the left part, $V_{D,n}$ is a filtered hyperinterpolation function constructed from a data set $D$ from the target function $f^*$. 
We show that the distance between $f^*$ and $V_{D,n}$ is approximately equal to the distance between $f^*$ and $f^\ast_{\Pi_n}$, which is the optimal approximation in the space $\Pi_n$. 
In the right part, $V_{D,n}^{(m)}$ is a weighted average of the individual filtered hyperinterpolations $V_{D_j,n}$ obtained from multiple datasets sampled from the target function $f^*$. 
Here again, the distance between $f^*$ and $V_{D,n}^{(m)}$ is approximately equal to the distance between $f^*$ and its optimal approximation $f^\ast_{\Pi_n}$ in $\Pi_n$. 
}\label{fig:learning_of_dfl}
\end{figure}
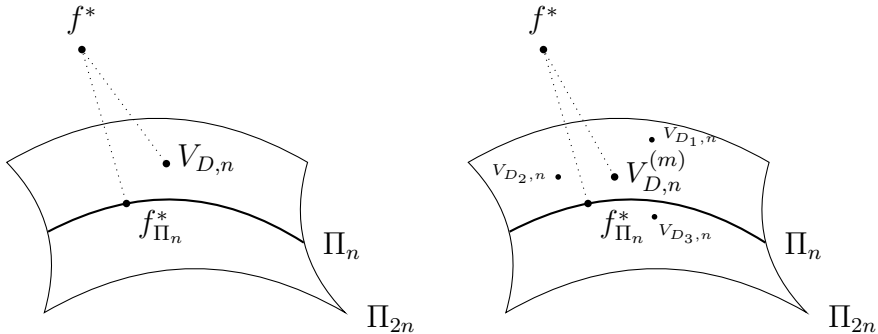

Filtered hyperinterpolation was introduced by \cite{SlWo2012} on the two-sphere $\mathbb{S}^2$, which is a form of filtered polynomial approximation method motivated by hyperinterpolation \citep{Sloan1995}. 
Hyperinterpolation uses a Fourier expansion where the integral for the Fourier coefficients is approximated by numerical integration with a quadrature rule. The filtered hyperinterpolation adopts a similar strategy as hyperinterpolation but uses a filter to modify the Fourier expansion. 
The filter is a restriction on the eigenvalues of the basis functions. 
Effectively this restricts the capacity of the approximation class and yields a reproducing property for polynomials of a certain degree specified by the filter. It has some similarities to kernel methods. 
Filtering improves the approximation accuracy of plain hyperinterpolation for noiseless data that is sampled deterministically \citep{HeSl2006}. With appropriate choice of filter, the filtered hyperinterpolation achieves the best approximation by polynomials of a given degree depending on the amount of data (see Section~\ref{sec:ndfh_clean}). 
As shown in the left part of Figure~\ref{fig:learning_of_dfl}, one aims at finding the closest approximation of $f^*$ within the polynomial space $\Pi_n$ on the manifold $\mfd$, which, nevertheless, is difficult to achieve. The filtered hyperinterpolation is an approximator $V_{D,n}$ constructed from data $D=(\bx_i, y_i)_{i=1}^N$ which lies in a slightly larger polynomial space $\Pi_{2n}$ and whose distance to $f^\ast$ is very close to the distance between $f^\ast$ and $\Pi_n$. 

Motivated by the problem of handling massive amounts of data, we propose a distributed computational strategy based on filtered hyperinterpolation. 
As shown in the right part of Figure~\ref{fig:learning_of_dfl}, we can split estimation task of filtered hyperinterpolation into multiple servers $j=1,\ldots,m$, each of which computes a filtered hyperinterpolation $V_{D_j,n}$,  for a small subset $D_j$ of all the training data. 
It consists of creating a filtered expansion in terms of eigenfunctions of the manifold to best-fit the corresponding fraction of the training data set. 
The ``best-fit'' means that the local servers can achieve best approximation for noisy data $y_i=f^*(\bx_i)+\epsilon_i$, $i=1,\ldots, N$, for any continuous function $f^*\colon\mathcal{M}\to\mathbb{R}$ on the manifold and independent bounded noise $\epsilon_i$. 
The central processor then takes a weighted average of the filtered hyperinterpolations obtained in the local servers to synthesize as a global estimator $V_{D,n}^{(m)}$. We call the global estimator the distributed filtered hyperinterpolation. 

The remaining of the paper is organized as follows. 
In Section~\ref{sec:preliminaries}, we introduce the main mathematical settings and notation. Then we proceed with the study of non-distributed and distributed filtered hyperinterpolation on manifolds, for which we derive upper bounds on the error. 
Our bounds depend on 
1) the dimension $d$ of the manifold and the smoothness $r$ of the Sobolev space that contains the target function, 
2) the degree $n$ of the approximating polynomials, 
which is tied to the number $N$ of available data points, 3) the smoothness of the filter, 
4) the presence of noise in the output data points. 
Here we base the analysis on properties of the quadrature formulas, which we couple with the arrangement of the input data points (deterministic or random). For the deterministic case, we require the quadrature rule has polynomial exactness of degree $3n-1$; for the random case, the condition that the volume measure on the manifold controls the distribution of the sampling points. 

In Section~\ref{sec:ndfh} we study non-distributed filtered interpolation on manifolds. We obtain an error bound $\bigo{}{N^{-r/d}}$
for the noiseless setting on general manifolds (see Theorem~\ref{thm:nondfh_clean_det}). This result generalizes the same bound that was previously obtained on the sphere \citep{WaSl2017}. 
Since the bound on the sphere is optimal, the new bound is also optimal. 
We further study learning with noisy output data. The error bound for the noisy case is $\bigo{}{N^{-2r/(2r +d)}}$. 
Due to the impact of the noise, it does not entirely reduce to the error bound of the noiseless case. To the best of our knowledge, this is the first error upper bound for noisy learning on general Riemannian manifolds. 
The optimality of this bound remains open at this point. 
 
In Section~\ref{sec:dfh} we study distributed learning. 
We obtain similar rates of convergence as in the non-distributed setting, provided the number of servers satisfies a certain upper bound in terms of the total amount of data. 
As it turns out, the distributed estimator has the same convergence rate as the non-distributed estimator for a class of functions with given smoothness. Compared with the clean data case, the distributed filtered hyperinterpolation with noisy data has slightly lower convergence order than the non-distributed. See Theorems~\ref{thm:dfh_noisy_det} and \ref{thm:dfh_ran_noisy}. 
 
Section~\ref{sec:example} illustrates definitions, methods, and convergence results on a concrete numerical example. 
Section~\ref{sec:comparison} summarizes and compares the convergence rates of the different methods and settings (see Table~\ref{tab:comparison}). 
It also presents a concise description of the implementation (see Algorithm~\ref{algorithm}). 
All the proofs are deferred to Appendix~\ref{sec:proofs}. The proofs utilize the wavelet decomposition of filtered hyperinterpolation,  Marcinkiewicz-Zygmund inequality, Nikolski\^{\i}-type inequality on a manifold, bounds of best approximation on a manifold, and concentration inequality, estimates of covering number and bounds of sampling operators from learning theory.
We also show a table of the notations used throughout the article in Appendix~\ref{app:notations} for readers' reference.

\section{Preliminaries on approximation on manifolds}
\label{sec:preliminaries}
In this section, we discuss $L_p$ and Sobolev spaces of functions on manifolds, assumptions on the manifolds and embedding theorems to the space of continuous functions. 

We start with a brief description of $L_p$ spaces and norms. 
Let $\mfd$ be a compact and smooth Riemannian manifold of dimension $d\ge1$ with smooth or empty boundary and Riemannian measure $\memf$ normalized to have the total volume $\mu(\mfd)=1$. 
For $1\le p<\infty$, let $\Lpm{p}=\Lpm[\mfd,\mu]{p}$ be the complex-valued $\mathbb{L}_{p}$-function space with respect to the measure $\memf$ on $\mfd$, endowed with the $\mathbb{L}_{p}$ norm
\begin{equation*}
   \norm{f}{\Lpm{p}}:=\left\{\int_{\mfd}|f(\bx)|^{p}\dmf{x}\right\}^{1/p},\quad f\in \Lpm{p}.
\end{equation*}
For $p=\infty$, let $\Lpm{\infty}:=\contm$ be the space of continuous functions on $\mfd$ with norm
\begin{equation*}
  \norm{f}{\Lpm{\infty}}:= \sup_{\bx\in \mfd}|f(\bx)|,\quad f\in \contm.
\end{equation*}
We will write $\norm{f(\bx)}{\Lpm{p},\bx}=\norm{f}{\Lpm{p}}$ to indicate the variable for integration when necessary.
For $p=2$, $\Lpm{2}$ is a Hilbert space with inner product $\InnerL{f,g}:=\int_{\mfd}f(\bx)\conj{g(\bx)}\dmf{x}$, $f,g\in\Lpm{2}$, where $\conj{g}$ is the complex conjugate to $g$.

\subsection{Diffusion polynomial space} 
Diffusion polynomials are a generalization of regular polynomials. We will use them to construct approximations of real-valued functions on manifolds. Let $\N:=\{1,2,\dots\}$ be the set of positive integers and let $\Nz=\N\cup \{0\}$. Let $\LBm$ be the \emph{Laplace-Beltrami operator} on $\mfd$, which has a sequence of eigenvalues $\{\eigvm\}_{\ell\in\N}$ and a corresponding sequence of orthonormal eigenfunctions $\{\eigfm\in\Lpm{2}\: |\: \LBm\eigfm = -\eigvm^{2}\: \eigfm,\; \ell\in\N\}$. We let $\eigvm[0]:=0$ and $\eigfm[0]:=1$. 
For $n\in\Nz$, the span $\polyspm:=\spann\{\eigfm| \eigvm\le n\}$ is called the diffusion polynomial space of degree $n$ on $\mfd$, and an element of $\polyspm$ is called a diffusion polynomial of degree $n$. In the following, we will refer to diffusion polynomials simply as polynomials. 

Let $\dist{\bx,\by}$ be the geodesic distance of points $\bx$ and $\by$ induced by the Riemannian metric on $\mfd$.
For $\bx\in\mfd$ and $\alpha,\beta>0$, let $\ball{\bx,\alpha}:=\{\by\in\mfd|\dist{\bx,\by}\le \alpha\}$ be the ball with center $\bx$ and radius $\alpha$, and let $\ball{\bx,\beta,\beta+\alpha}:= \ball{\bx,\beta+\alpha}-\ball{\bx,\alpha}$ and $\ball{\bx,0,\alpha}:=\ball{\bx,\alpha}$. 
We make the following assumptions for the measure $\memf$ and the eigenfunctions of $\LBm$ on $\mfd$. 
The first is a standard assumption about the regularity of the measure on the manifold. 
\begin{assumption}[Volume of ball]
\label{assump:ball.vol}
There exists a positive constant $c$ depending only upon the measure $\memf$ and the dimension $d$ such that for all $\alpha>0$ and $\bx\in\mfd$,
    \begin{equation*}
        \memf\left(\ball{\bx,\alpha}\right) = c \:\alpha^{d}.
    \end{equation*}
\end{assumption}
The second is an assumption stating that the space of polynomials is closed under multiplication. 
\begin{assumption}[Product of eigenfunctions]\label{assump:poly}
For $\ell,\ell'\in \Nz$, the product of eigenfunctions $\eigfm$, $\eigfm[\ell']$ for the Laplace-Beltrami operator $\LBm$ on $\mfd$ is a polynomial of degree $\ell+\ell'$, i.e. $\eigfm\: \eigfm[\ell']\in \polyspm[\ell+\ell']$.
\end{assumption}

Assumption~\ref{assump:poly} implies that the product $P_{\ell}P_{\ell'}$ of two polynomials $P_{\ell}\in\polyspm[\ell]$ and $P_{\ell'}\in\polyspm[\ell']$ of degrees $\ell\in\Nz$ and $\ell'\in\Nz$, respectively, 
is a polynomial of degree $\ell+\ell'$. 
Assumptions~\ref{assump:ball.vol} and \ref{assump:poly} are satisfied by typical manifolds, such as hypercubes $[0,1]^{d}$, unit spheres and balls in real or complex Euclidean coordinate spaces \citep{HeSlWo2010,dai2013approximation}, flat tori $\torus{d}$, $d\ge1$, and Grassmannians \citep{Breger2017,breger2017quasi}, simplexes in $\Rd$ \citep{wang2018analysis,xu2010fourier}, with Lebesgue measures induced by the corresponding Riemannian metric, and also graph (a discrete manifold) which Lebesgue is the atom measure on graph nodes \citep{wang2019tight}.

\subsection{Generalized Sobolev spaces}
We give a brief introduction to the Sobolev spaces on a Riemannian manifold $\mfd$. 
The \emph{Fourier coefficients} for $f$ in $\Lpm{1}$ are
\begin{equation*}
 \Fcoem{f}:=\int_{\mfd}f(\bx)\conj{\eigfm(\bx)} \dmf{x}, \;\; \ell=0,1,\dots.
\end{equation*}
For $s>0$, the \emph{generalized Sobolev space} $\sobm{p}{s}$ may be defined as the set of all functions $f\in \Lpm{p}$ satisfying
$\sum_{\ell=0}^{\infty}(1+\eigvm)^{s/2}\Fcoem{f}\: \eigfm\: \in\: \Lpm{p}$.
The Sobolev space $\sobm{p}{s}$ forms a Banach space with norm
\begin{equation*}
  \norm{f}{\sobm{p}{s}}:=\normB{\sum_{\ell=0}^{\infty}(1+\eigvm)^{s/2} \Fcoem{f}\: \eigfm}{\Lpm{p}}.
\end{equation*}
We let $\sobm{p}{0}:=\Lpm{p}$.

In the context of numerical analysis, we need to use the following Lemma~\ref{lm:embed.sob.C} which is an embedding theorem of Sobolev space into the space of continuous functions on a manifold, see e.g. \citet[Section~2.7]{Aubin1998}. It guarantees that any function in the Sobolev space has a representation by a continuous function so that the numerical integration is valid and the quadrature rule can be applied. 

\begin{lemma}\label{lm:embed.sob.C} Let $d\ge1$ and $\mfd$ be a compact Riemannian manifold of dimension $d$. The Sobolev space $\sobm{p}{s}$ is continuously embedded into $\contm$ if $s>d/p$.
\end{lemma}

\subsection{Filtered approximation on manifolds}\label{sec:fiapprox}
This section defines the filtered polynomial approximation on a compact Riemannian manifold $\mfd$ in terms of the eigenfunctions of the Laplace-Beltrami operator $\LBm$ on $\mfd$. Given a target function $f^\ast\in\Lpm{p}$ with $1\le p\le\infty$, the filtered polynomial approximation converges to functions in $\Lpm{p}$ 
as the degree $n$ tends to infinity. 

\paragraph{Filter} A real-valued continuous compactly supported function on $\Rplus$ is called a \emph{filter}. Without loss of generality, we will only consider filters with support a subset of $[0,2]$.
In this paper, we focus on the following function $\fiH$ on $\Rplus$ as the filter.
\begin{definition}[Filter $H$]\label{def:fiH}
    Let $\fiH$ be a filter on $\Rplus$ satisfying $\fiH(t)=1$, $0\le t\le 1$; $\fiH(t)=0$, $t\ge2$, and $\fiH\in \CkRp$ for some $\fis\in \N$.
\end{definition}

\begin{definition}[Filtered kernel]\label{def:fil.fil.ker}
A filtered kernel of degree $n$ for $n\in \N$ on $\mfd$ with filter $\fiH$ is defined by
\begin{equation}\label{eq:fiker}
  \vdh{n}(\bx,\by):=\vdh{n,\fiH}(\bx,\by):=\sum_{\ell=0}^{\infty}\fiH\Bigl(\frac{\eigvm}{n}\Bigr)\:\eigfm(\bx)\conj{\eigfm(\by)}.
\end{equation}
Here $\lambda_\ell$ and $\phi_\ell$ are eigenvalues and eigenfunctions of the Laplace-Beltrami operator on $\mathcal{M}$.
\end{definition}

For a kernel $G:\mfd\times\mfd\to\Rone$ and $f\in\Lpm{1}$, the \emph{convolution} of $f$ with $G$ is defined as
\begin{equation}
    (G\conv f)(\bx):= \int_{\mfd}G(\bx,\PT{z})f(\PT{z})\dmf{z},\quad x\in\mfd.
\end{equation}

\begin{definition}[Filtered approximation]\label{defn:fiapp}
We can define a \emph{filtered approximation} $\fiapprox{n}$ on $\Lpm{1}$ as an integral operator with the filtered kernel $\vdh{n,\fiH}(\cdot,\cdot)$: for $f\in \Lpm{1}$ and $\bx\in\mfd$,
\begin{equation}\label{eq:fiapprox}
  \fiapprox{n}(f;\bx) := \fiapprox{n,\fiH}(f;\bx):= (\vdh{n,\fiH}\conv f)(\bx)
  :=\int_{\mfd}\vdh{n,\fiH}(\bx,\PT{z})f(\PT{z})\dmf{z}.
\end{equation}
Note that for $n=0$ this is just the integral of $f$. By \eqref{eq:fiker} and \eqref{eq:fiapprox},
\begin{equation*}
  \fiapprox{n}(f) = \sum_{\ell=0}^{\infty}\fiH\left(\frac{\eigvm}{n}\right)\Fcoem{f}\:\eigfm, \quad f\in\Lpm{1}.
\end{equation*}
\end{definition}

The following lemma, as given by \citet[Theorem~4.1]{MaMh2008}, shows that a filtered kernel is highly localized when the filter is sufficiently smooth.
\begin{lemma}\label{lem:localization} Let $d\ge1$. Let $\mfd$ be a compact Riemannian manifold of dimension $d$. Let $\fiH$ be a filter in $\CkRp$ with $\fis\ge d+1$. Then, for $n\ge1$,
\begin{equation}\label{eq:fiker.localisation}
    \bigl|\vdh{n}(\bx,\by)\bigr| \le \frac{c\: n^{d}}{(1+n\dist{\bx,\by})^{\fis}},\quad \bx,\by\in\mfd,
\end{equation}
where the constant $c$ depends only on $d,\fiH$ and $\fis$ and $\dist{\bx,\by}$ is the geodesic distance between $\bx$ and $\by$.
\end{lemma}
By \eqref{eq:fiker.localisation}, if $\by$ is not close to $\bx$, and $|K_n(\bx,\by)|$ decays to zero with rate $n^{\kappa-d}$. It means given $\bx$, the kernel $|K_n(\bx,\cdot)|$ is concentrated on a small neighbourhood of $\bx$, although it is supported on the whole manifold.
This localization is essential to the boundedness of the filtered approximation operator.
\begin{remark}
    For sphere $\mfd$ case, the above lemma for $p=1$ was proved by \cite{WaLeSlWo2017} (see
also \cite{NaPeWa2006} for $\kappa\geq d+1$); the case $p> 1$ can
be obtained from the case $p=1$ with the fact that $K_n\in\Pi_{2n}^d$ and the Nikolski\^{\i} inequality for spherical
polynomials \citep{MhNaWa1999}.
\end{remark}

Lemma~\ref{lem:localization} by \citet[Eq.~6.28]{MaMh2008} implies the following estimate for the $L_p$-norm of the filtered kernel.
\begin{lemma}\label{lem:fikerL1} Let $d\ge1$ and $1\leq p\leq \infty$. Let $\mfd$ be a compact Riemannian manifold of dimension $d$. Let $\fiH$ be a filter in $\CkRp$ with $\fis\ge d+1$. Then, for $n\geq1$ and $\bx\in\mfd$,
\begin{equation*}
    \normb{\vdh{n}(\cdot,\bx)}{\Lpm{p}} \le c\: n^{d(1-1/p)},
\end{equation*}
where the constant $c$ depends only on $d, p,\fiH$ and $\fis$.
\end{lemma}

Using the interpolation theorem with \eqref{eq:fiapprox} gives
\begin{equation*}
    \norm{\fiapprox{n}(f)}{\Lpm{p}} \le \max_{\bx\in\mfd}\norm{K_{n}(\cdot,\bx)}{\Lpm{1}}\:\norm{f}{\Lpm{p}}.
\end{equation*}
This with Lemma~\ref{lem:fikerL1} implies the following boundedness of the filtered approximation on $\Lpm{p}$.
\begin{theorem}\label{thm:fiapprox.Bd} Let $d\ge1$, $1\le p\le\infty$. Let $\mfd$ be a compact Riemannian manifold of dimension $d$. Let $\fiH$ be a filter in $\CkRp$ with $\fis\ge d+1$. Then for $n\geq1$, the operator norm of $\fiapprox{n}$ on $\Lpm{p}$
\begin{equation*}
    \norm{\fiapprox{n}}{\ptop} \le c,
\end{equation*}
where the constant $c$ depends only on $d,\fiH$ and $\fis$.
\end{theorem}

\paragraph{Polynomial space and best approximation} Let $\Pi_n := \operatorname{span}\{\phi_1,\ldots, \phi_n\}$ be the \emph{(diffusion) polynomial space} of degree $n$ on manifold $\mfd$. Given $1\le p\le\infty$ and $n\in \N$, let $\bestapprox{n}{f}:=E_{n}(\Lpm{p};f):=\inf\bigl\{\normb{f-P}{\Lpm{p}} | P\in \polyspm\bigr\}$ be the \emph{best approximation} of degree $n$ for $f\in\Lpm{p}$. 
Since $\cup_{n=0}^{\infty}\polyspm$ is dense in $\Lpm{p}$, $\bestapprox{n}{f}$ goes to zero as $n\to\infty$.

The following theorem proves the convergence error for the filtered approximation of $f\in\Lpm{p}$.
\begin{theorem}\label{thm:fiapprox.err.Lp} Let $d\ge1$, $1\le p\le\infty$ and $\mfd$ be a compact Riemannian manifold of dimension $d$. Let $\fiapprox{n}$ be the filtered approximation with filter $\fiH$ given by Definition~\ref{def:fiH} satisfying $\fis\ge d+1$. Then, for $f\in\Lpm{p}$ and $n\in\Nz$,
\begin{equation*}
    \normb{f-\fiapprox{n}(f)}{\Lpm{p}} \le c\: \bestapprox{n}{f},
\end{equation*}
where the constant $c$ depends only on $d$, $\fiH$ and $\fis$.
\end{theorem}
\begin{remark}
For $L_p([0,1])$, the case of filtered approximation with an appropriate filter reduces to a classic result of de la Vall\'{e}e-Poussin approximation \citep{vallee1919lecons}. 
\cite{stein1957interpolation} proved in a general context the convergence of de La Vall\'{e}e-Pousson approximation to the target function.
The sphere case of Theorem~\ref{thm:fiapprox.err.Lp} was proved by \cite{rustamov1994on,LeMh2008,sloan2011polynomial}.
\end{remark}

The following lemma gives the convergence error of the best approximation for $f\in\sobm{p}{s}$, see \cite{MaMh2008}.
\begin{lemma}\label{lem:best.approx.sob.mfd} Let $d\ge1$, $1\le p\le \infty$, $s>0$, and $\mfd$ be a compact Riemannian manifold of dimension $d$. For $f\in \sobm{p}{s}$ and $n\in\N$,
\begin{equation*}
    \bestapprox{n}{f} \le c\: n^{-s}\:\norm{f}{\sobm{p}{s}},
\end{equation*}
where the constant $c$ depends only on $d$, $p$ and $s$.
\end{lemma}

Theorem~\ref{thm:fiapprox.err.Lp} and Lemma~\ref{lem:best.approx.sob.mfd} imply the following convergence order for the filtered approximation of a smooth function on a compact Riemannian manifold.
\begin{theorem}\label{thm:fiapprox.err.Wp} Let $d\ge1$, $1\le p\le\infty$ and $\mfd$ be a compact Riemannian manifold of dimension $d$. Let $\fiapprox{n}$ be the filtered approximation with filter $\fiH$ given by Definition~\ref{def:fiH} satisfying $\fis\ge d+1$. Then, for $f\in\sobm{p}{s}$ and $n\in\N$,
\begin{equation*}
    \normb{f-\fihyper{n}(f)}{\Lpm{p}} \le c\: n^{-s}\norm{f}{\sobm{p}{s}},
\end{equation*}
where the constant $c$ depends only on $d$, $p$, $s$, $\fiH$ and $\fis$.
\end{theorem}

In the following Sections~\ref{sec:ndfh} and \ref{sec:dfh}, we will study the non-distributed and distributed filtered hyperinterpolation's which use single and multiple servers to find a global estimator respectively. For both non-distributed and distributed learning by filtered hyperinterpolation, we need to take account of the data type (noise or noiseless) and the quadrature point type (deterministic or random). There are in total 8 cases for which we have to treat separately.

\section{Non-distributed filtered hyperinterpolation on manifolds}
\label{sec:ndfh}
In this section, we study the non-distributed version of filtered hyperinterpolation (NDFH) on a manifold. 
We consider the cases when the data is either clean or noisy, and the input samples are either deterministic or random. It turns out that the NDFH for clean data achieves the optimal convergence order of the approximation error, while noise on the data would reduce the convergence order. 

Filtered hyperinterpolation is a special type of regression, and the primary tool that we will use. 
Within this approach, as introduced in Definition~\ref{defn:fiapp}, a target function $f^\ast$ is approximated by the filtered polynomial approximation
\begin{equation}\label{eq:fihyper1}
\sum_{\ell=0}^{\infty} H(\lambda_{\ell}/n) \hfs_{\ell}  \phi_{\ell}(\bx). 
\end{equation}
Here $H$ is a filter for the eigenvalues $\lambda_{\ell}$ to eigenfunctions $\phi_{\ell}$, and $\hfs_{\ell}$ are the Fourier coefficients. 
The Fourier coefficients cannot be computed in practice, because they would require to integrate the unknown target function. Instead, they are estimated from samples. This estimation is conducted via a quadrature formula, 
\begin{equation*}
    \hfs_{\ell} = \langle f^*, 
    \phi_{\ell}\rangle = \int_{\mathcal{M}} f^*(\by) \overline{\phi_{\ell}(\by)} \dmf{y} \approx \sum_{i=1}^N w_i f^*(\bx_i)\overline{\phi_{\ell}(\bx_i)}.
\end{equation*} 
We rewrite \eqref{eq:fihyper1} as
\begin{equation*}
\sum_{\ell=0}^{\infty} H(\lambda_{\ell}/n) \hfs_{\ell}  \phi_{\ell}(\bx)
\approx
\sum_{\ell=0}^{\infty} H(\lambda_{\ell}/n) \phi_{\ell}(\bx)\sum_{i=1}^N w_i f^*(\bx_i)\overline{\phi_{\ell}(\bx_i)}.
\end{equation*}
After rearranging, our approximation takes the form
\begin{equation}\label{eq:fihyper2}
\sum_{i=1}^N w_i f^*(\bx_i) K_n(\bx_i,\bx), 
\end{equation}
which is a weighted sum of kernels $K_n(\bx_i,\bx)=\sum_{\ell=0}^{\infty}H(\lambda_\ell/n)\overline{\phi_{\ell}(\bx_i)}\phi_\ell(\bx)$ centered at the data locations $\bx_i$. In practice, the estimator of \eqref{eq:fihyper2} is scaled by the observed values $y_i$ instead of $f^*(\bx_i)$. 

In the following, we define the \emph{(non-distributed) filtered hyperinterpolation (approximation)} on a compact Riemannian manifold $\mfd$ for a data set $D$. 
Besides the traditional deterministic quadrature rule, we also consider the filtered hyperinterpolation with random quadrature rule where the quadrature points are distributed with some probability measure on the manifold. We first introduce some notion about data and quadrature rule. 

\paragraph{Data} Let $\mfd$ be a compact Riemannian manifold of dimension $d$ for $d\geq1$. 
A \emph{data set} $D=\{(\bx_i,y_i)\}_{i=1}^{N}$, $N=|D|$ on the manifold
$\mfd$ is a set of pairs of points $\Lambda_{D}:=\{\mathbf
x_i\}_{i=1}^{N}$ on the manifold and real numbers $y_i$. 
Elements of
$D$ are called \emph{data points}. 
The points $\bx_i$ of
$\Lambda_{D}$ are called \emph{input samples}. 
The $y_i$ are called \emph{data values}. A continuous function $f^*$ on the manifold is called an (ideal) \emph{target function} for data $D$ if 
\begin{equation}\label{eq:datavalue}
    y_i=f^*(\bx_i)+\epsilon_i,\quad i=1,\dots,|D|
\end{equation}
for \emph{noises} $\epsilon_i$.
\paragraph{Deterministic and random sampling} In this paper, we consider two types of input samples depending on whether they are randomly sampled: 
the \emph{deterministic sampling} and \emph{random sampling}. The data $D$ has random sampling if $\bx_i$ are randomly chosen with respect some probability measure on $\mfd$. 
In contrast, $D$ has deterministic sampling if the $\bx_i$  are fixed. 
\paragraph{Noisy and noiseless data} We also distinguish data types by its data values $y_i$. We say $D$ is \emph{noiseless data or clean data} if $y_i$ is equal to the function value of the associated (ideal) target function value $f^*(\bx_i)$ (that is, the noises $\epsilon_i\equiv0$). We say $D$ is \emph{noisy data} if the noises $\epsilon_i$ in \eqref{eq:datavalue} are non-zero.

\paragraph{Quadrature rule}
A set 
\begin{equation*}
	\QH=\{(\wH,\pH{i})|\wH\in\Rone, \pH{i}\in\mfd, i=1,\dots,N\}
\end{equation*}
is said to be a \emph{quadrature rule} for numerical integration on $\mfd$. We say $\QH$ is a \emph{positive quadrature rule} if all weights $\wH>0$, $i=1,\dots,N$. In this paper, we only consider positive quadrature rules. 

\begin{definition}[Non-distributed filtered hyperinterpolation]\label{defn:nondfh_det}
Let $D=\{(\bx_i,y_i)\}_{i=1}^{|D|}$ be a data set on compact Riemannian manifold $\mfd$, $\QH=\{(\wH,\bx_{i})\}_{i=1}^{|D|}$ a positive quadrature rule on $\mfd$ and $\fiH$ be a filter in Definition~\ref{def:fil.fil.ker} on $\Rplus$. For $n\in \N$, the non-distributed filtered hyperinterpolation (NDFH) for data $D$ and quadrature rule $\QH$ is 
\begin{equation}\label{eq:VDn}
    \fihyper{D,n}(\bx) := \fihyper{D,n,\fiH,\QH}(\bx) 
    :=\sum_{i=1}^{|D|} \wH y_i\vdh{n,\fiH}(\bx,\bx_i).
\end{equation}
If we let $D^*:=D^*(f^*):=\{(\bx_i,f^*(\bx_i))\}_{i=1}^N$ be the noiseless data for the ideal target function $f^*$ and data $D$, then \eqref{eq:VDn} becomes
\begin{equation}\label{eq:VDnf}
    \fihyper{D^*,n}(\bx):=\fihyper{D^*,n}(f^*,\bx) := \sum_{i=1}^{|D|} \wH f^*(\bx_i)\vdh{n,\fiH}(\bx,\bx_i).
\end{equation}
\end{definition}
We call $\fihyper{D^*,n}(f^*)$ \emph{non-distributed filtered hyperinterpolation (NDFH) for clean data set or for quadrature rule $\QH$}, for the function $f^*$.
\begin{remark}
Non-distributed filtered hyperinterpolation on the sphere was studied by \cite{SlWo2012}.
\end{remark}

\subsection{Non-distributed filtered hyperinterpolation for clean data}\label{sec:ndfh_clean}
We first assume that we have a quadrature rule that has polynomial exactness. That is, the weighted sum by the quadrature rule can recover the integral for polynomials on manifolds. 
The non-distributed filtered hyperinterpolation with polynomial-exact quadrature rule can reach the same optimal convergence order as the filtered approximation in Section~\ref{sec:fiapprox}, and the convergence rate is optimal.

Let $\ell\in\Nz$. A positive quadrature rule $\QH:=\gQ:=\{(\wH,\pH{i})\}_{i=1}^{N}$ on $\mfd$ is said to be \emph{exact} for degree $\ell$ if for all polynomials $P\in\polyspm[\ell]$,
\begin{equation*}
    \int_{\mfd}P(\bx)\dmf{x} =  \sum_{i=1}^{N} \wH P(\pH{i}).
\end{equation*}
That the quadrature is exact for polynomials is a strong assumption, as the optimal-order number of points is $\bigo{}{N^d}$ in typical examples of manifolds, see e.g. \cite{HeSlWo2010,Cools2003}.

The following lemma shows that the filtered hyperinterpolation $\fihyper{D,n}$ with filter $\fiH$ given by Definition~\ref{def:fiH} reproduces polynomials of degree up to $n$ if the associated quadrature rule $\QH$ is exact for degree $3n-1$.
\begin{lemma}\label{lm:fihyper.reproduce.p} Let $n\in\Nz$ and $\mfd$ be a $d$-dimensional compact Riemannian manifold. Let $\QH := \{(\wH,\pH{i})\}_{i=1}^{N}$ be a positive quadrature rule on $\mfd$ exact for polynomials of degree up to $3n-1$ and let $\fihyper{D^*,n}$ be a non-distributed filtered hyperinterpolation on $\mfd$ for quadrature rule $\QH$ with filter $\fiH$ given by Definition~\ref{def:fiH}. Then,
\begin{equation*}
    \fihyper{D^*,n}(P) = P,\quad P\in\polyspm.
\end{equation*}
\end{lemma}

\begin{theorem}[NDFH with clean data and deterministic samples]
\label{thm:nondfh_clean_det}
Let $d\ge1$, $1\le p\le\infty$ and $n\ge1$. Let $\mfd$ be a compact Riemannian manifold of dimension $d$. Let $\fiH$ be a filter given by Definition~\ref{def:fiH} with $\fis\ge d+1$ and $\QH$ be a positive quadrature rule exact for polynomials of degree up to $3n-1$. Then, for $f\in\sobm{p}{s}$ with $s>d/p$, the NDFH for the quadrature rule $\QH$ has the error upper bounded by
\begin{equation}\label{eq:fihyper.Lp.err.Wp}
    \normb{f-\fiapprox{D^*,n}(f^*)}{\Lpm{p}} \le c\: n^{-s}\norm{f}{\sobm{p}{s}},
\end{equation}
where the constant $c$ depends only on $d$, $p$, $s$, $\fiH$ and $\fis$.
\end{theorem}

From the perspective of information-based complexity it is interesting to observe that if the target function $f^\ast$ is in the Sobolev space $\sobm{p}{s}$, $s>0$, the convergence rate is optimal in the sense of optimal recovery. This is due to that on a real unit sphere when one uses optimal-order number of points $N =\bigo{}{n^d}$, the order $n^{-s}=N^{-s/d}$ in \eqref{eq:fihyper.Lp.err.Wp} is optimal, as proved by \cite{WaSl2017,WaWa2016}. 
Theorem~\ref{thm:nondfh_clean_det} can be viewed as the non-distributed filtered hyperinterpolation for clean data, where the estimator uses the whole data set in one machine.

We now introduce the (non-distributed) filtered hyperinterpolation for clean data with random sampling.
We say a data set $D$ has \emph{random sampling} (with distribution $\nu$) if the sampling points $\bx_i$ of $D$ are independent and identically distributed (i.i.d.) random points with distribution $\nu$ on $\mfd$. 
To construct the filtered hyperinterpolation for random sampling, we need the following lemma, which shows that there exist $N$ quadrature weights given $N$ i.i.d.\ random points $\bx_i$ such that the resulting quadrature rule is exact for polynomials for degree $n$ with high probability. 
For $1\leq p\leq \infty$, let $L_{p,\nu}(\mfd)$ be $L_p$ space on manifold $\mfd$ with respect to probablity measure $\nu$.

\begin{lemma}[Quadrature rule for random samples]\label{lem:random_quadr}
 For $N\geq2$, let $X_N=\{\bx_i\}_{i=1}^{N}$ be a set of $N$ i.i.d.~random points on $\mfd$ with distribution $\nu$, where $\nu$ satisfies
\begin{equation}\label{eq:condition on distribution}
         \|f\|_{L_1(\mfd)} \leq c \|f\|_{L_{1,\nu}(\mfd)}\quad\forall
         f\in L_1(\mfd)\cap L_{1,\nu}(\mfd), 
\end{equation}
 for a positive absolute constant $c$. Then, for integer $n$ satisfying $N/n^{2d}>c$ for sufficiently large constant $c$, there exists a quadrature rule $\{(\bx_i,w_{i,n})\}_{i=1}^{N}$ such that
\begin{equation*}
    \int_{\mfd}P_{n}(\bx)\mathrm{d}\nu(\bx)=
    \sum_{i=1}^{N}w_{i,n}P_{n}(\bx_i) \quad\forall P_n\in
    \Pi_n^d
\end{equation*}
holds, and $\sum_{i=1}^{N}| {w_{i,n}}|^2\leq2/N$, with confidence at least $1-4\exp\left\{-CN/n^d\right\}$,
where $C$ is a constant depending only on $c_1$ and $d$.
\end{lemma}
We call the set $\{(\bx_i,w_{i,n})\}_{i=1}^{N}$ \emph{quadrature rule for random samples} on the manifold $\mfd$ for measure $\nu$.

The following theorem gives the approximation error of the non-distribured filtered hyperinterpolation with clean data and random sampling for sufficiently smooth functions. Here, we want to obtain an estimated value of the expected error and take the expectation over the distribution of the data $P(X)P(Y|X)$.
\begin{theorem}[NDFH with clean data and random samples]\label{thm:ndfh_ran_clean}
Let $d\geq 2$ and $r>d/2$. Let the clean data set $D^*$ with i.i.d. random sampling points on $\mfd$ and 
distribution $\nu$ satisfying \eqref{eq:condition on distribution}. Given some $\tau$, $0<\tau\leq d$, for $cn^{d+\tau}\leq |D^*|\leq c' n^{2d}$ with two positive constants $c,c'$, the filtered hyperinterpolation $V_{D^*,n}$ for clean data set $D^*$ with target function $f^*\in \mathbb W_2^r(\mfd)$, as given by \eqref{eq:VDnf}, has
the approximation error
\begin{equation*}
    \mathbf{E}\left\{\|V_{D^*,n}-f^*\|_{L_2(\mfd)}^2\right\}\leq C |D^*|^{-r/d},
\end{equation*}
where $C$ is a constant independent of $|D^*|$.
\end{theorem}

Theorem~\ref{thm:ndfh_ran_clean} shows that the filtered
hyperinterpolation with random sampling for clean data can achieve the same optimal convergence rate as the filtered hyperinterpolation with deterministic sampling. We give the proof of Theorem \ref{thm:ndfh_ran_clean} in Section~\ref{appendix:ndfh_clean}.

\subsection{Non-distributed filtered hyperinterpolation for noisy data}\label{sec:ndfh_noisy}

In the following we describe non-distributed filtered hyperinterpolation with deterministic or random sampling for noisy data. The data $y_i$ are
the values of a function $f^*$ on $\mfd$ plus noise. Here we assume the noise be mean zero and bounded. To be precise, we let
\begin{equation}\label{eq:noisydats}
        y_i=f^*(\bx_i)+\epsilon_i, \quad \mathbf{E}[\epsilon_i]=0,\quad |\epsilon_i|\leq M \quad\forall
        i=1,\dots,|D|.
\end{equation}
The $D$ satisfying \eqref{eq:noisydats} is then called
\emph{noisy data set associated with $f^*$}. For real data, the $f^*$ is an unknown mapping from input to output.
We study the performance of the non-distributed filtered hyperinterpolation for a noisy data set $D$ whose data are stored in a sufficiently big machine.

We first consider the case where the locations of sampling points are fixed, which we call filtered hyperinterpolation with \emph{deterministic sampling}. 
The kernel $K_n$ provides a smoothing method for the function $f^*$ using data $D$. As we shall see below, the approximation error of this filtered hyperinterpolation has the convergence rate depending on the smoothness of function $f^*$. 
The following assumes that there exists a quadrature rule with $N$ nodes and $N$ ``almost equal'' weights which are exact for polynomials of degree approximately $N^{1/d}$. 
\begin{assumption}[Polynomial-exact quadrature]\label{assum:qrmfd}
Let $\mfd$ be a $d$-dimensional compact Riemannian manifold. For a point set $X_N:=\{\bx_1,\dots,\bx_N\}\subset\mfd$, there exist $N$ positive weights $\{w_j\}_{j=1}^N$ and constants $c_2$ and $c_3$ such that $0<w_j<c_2 N^{-1}$ and
\begin{equation}\label{eq:qrmfd}
    \int_{\mfd}f(\bx)\dmf{x} = \sum_{j=1}^{N}w_j f(\bx_j)\quad\forall f\in \polyspm[c_3 N^{1/d}]. 
    \end{equation}
\end{assumption}

\begin{remark}
For the sphere of any dimension, Assumption~\ref{assum:qrmfd} always holds \citep{MhNaWa2001}. 
In order to construct the quadrature rule for general Riemannian manifolds, one needs to find weights that make the worst case error vanish. This corresponds to solving a particular equation 
\begin{equation*}
    \sum_{i,j=1}^N \omega_i\omega_j \mathcal{K}(\bx_i,\bx_j)=0\quad 
\text{subject to $\sum_{i=1}^N \omega_i=1$}, 
\end{equation*} 
where $\mathcal{K}(\bx_i,\bx_j)$ is the reproducing kernel removing the constant $1$ of the Sobolev space $\mathbb{H}^s(\mfd):=\{f\in L_p(\mfd):\sum_{\ell=0}^{\infty}\widehat{f}_{\ell}\eigfm\in L_p(\mfd)\}$, given by $\mathcal{K}(\bx_i,\bx_j):=\sum_{\ell=1}^{\infty}(1+\lambda_\ell)^{s}\eigfm(\bx_i)\eigfm(\bx_j)$, where $\bx_i,\bx_j\in\mfd$. 
\end{remark}

The following theorem shows that the filtered hyperinterpolation $V_{D,n}$ can approximate $f^*$ well, provided that the support of the filtered kernel is appropriately tuned and Assumption~\ref{assum:qrmfd} holds.

\begin{theorem}[NDFH for noisy data and deterministic samples]
\label{thm:ndfh_noisy_det}
Let $d\geq 2$ and $r>d/2$. The sampling point set of the data set $D$ satisfies Assumption~\ref{assum:qrmfd}. Then, for $\frac{c_3}6 |D|^{1/(2r+d)} \leq n\leq \frac{c_3}2 |D|^{1/(2r+d)}$ with constant $c_3$ in \eqref{eq:qrmfd}, the filtered hyperinterpolation $V_{D,n}$ for noisy data set $D$ with target function $f^*\in \mathbb W_2^r(\mfd)$ satisfies
\begin{equation}\label{eq:error 1}
    \mathbf{E}\left\{\|V_{D,n}-f^*\|_{L_2(\mfd)}^2\right\} \leq C_1 |D|^{-2r/(2r+d)},
\end{equation}
where $C_1$ is a constant independent of $|D|$ and $n$.
\end{theorem}
Here, in contrast to Theorem~\ref{thm:nondfh_clean_det}, $y$ contains noise. The expectation in \eqref{eq:error 1} is with respect to the noise on $y$. The variance of the noise enters in $C_1$.

\begin{remark}
    Here the condition $r>d/2$ is the embedding condition such that any function in $\mathbb W_2^r(\mfd)$ has a representation of a continuous function on $\mfd$, which makes quadrature rule of filtered hyperinterpolation feasible for numerical computation.
\end{remark}

Theorem~\ref{thm:ndfh_noisy_det} illustrates that if
the scattered data $\Lambda_{D}$ has polynomial-exactness, and the support of the filter $\eta$ is appropriately chosen, then the filtered hyperinterpolation for noisy data set
$D$ can approximate sufficiently smooth target function
$f^*$ on the manifold in high precision in probablistic sense. By
\cite{GyKoKrWa2002}, the rate $|D|^{-2r/(2r+d)}$ in \eqref{eq:error 1}
cannot be essentially improved in the scenario of
\eqref{eq:noisydats}. Theorem~\ref{thm:ndfh_noisy_det} thus provides a feasibility analysis of the filtered
hyperinterpolation for manifold-structured data with random noise.

Now, we introduce the (non-distributed) filtered hyperinterpolation for noisy data with random sampling.
Let $D=\{(\bx_i,y_i)\}_{i=1}^{|D|}$ where the $\bx_i$ are i.i.d.\ random points with distribution $\nu$ on $\mfd$. 
The following theorem gives the approximation error of the non-distributed filtered hyperinterpolation for sufficiently smooth functions. Here, we want to get an estimated value of the expected error and take the expectation over the distribution $P(X)P(Y|X)$ of the data.
\begin{theorem}[NDFH for noisy data and random samples]
\label{thm:ndfh_noisy_ran}
Let $d\geq 2$ and $r>d/2$. Let the noisy data set $D$ take i.i.d.\ random sampling points on $\mfd$ with distribution $\nu$ satisfying \eqref{eq:condition on distribution}. For $n\asymp |D|^{1/(2r+d)}$, the filtered hyperinterpolation $V_{D,n}$ for noisy data set $D$ with target function $f^*\in \mathbb W_2^r(\mfd)$ has
the approximation error
\begin{equation*}
            \mathbf{E}\left\{\|V_{D,n}-f^*\|_{L_2(\mfd)}^2\right\}\leq C_3|D|^{-2r/(2r+d)},
\end{equation*}
where $C_3$ is a constant independent of $|D|$, and for two sequences $\{a_n\}_{n=1}^{\infty}, \{b_n\}_{n=1}^{\infty}$, $a_n\asymp b_n$ means that there exist  constants $c'$, $c$ such that $c' b_n\leq a_n\leq c b_n$.
\end{theorem}

Theorems~\ref{thm:ndfh_noisy_det} and
\ref{thm:ndfh_noisy_ran} show that the filtered
hyperinterpolation approximations with deterministic sampling and random sampling can achieve the same optimal convergence rate.
We give the proofs of Theorems~\ref{thm:ndfh_noisy_det} and \ref{thm:ndfh_noisy_ran} in Section \ref{appendix:ndfh_noisy}.

\section{Distributed filtered hyperinterpolation on manifolds}
\label{sec:dfh}
In this section, we describe the distributed learning by filtered hyperinterpolation for clean data with deterministic and random sampling's.
\paragraph{Distributed data sets} We say a large data set $D$ is \emph{distributively stored} in $m$
local servers if for $j=1,\dots,m$, $m\geq2$, the $j$th server contains a subset $D_j$ of $D$, and there is no common data between any pair of
servers, that is, $D_j\cap D_{j'}=\emptyset$ for $j\neq j'$, and
$D=\cup_{j=1}^m D_j$. The data sets $D_1,\dots,D_m$ are called
\emph{distributed data sets} of $D$. In this case, the filtered
hyperinterpolation $V_{D,n}$ which needs access to the
entire data set $D$ is infeasible. Instead, in this section, we
construct a \emph{distributed filtered hyperinterpolation for the
distributed data sets $\{D_j\}_{j=1}^{m}$ of $D$} by the divide and
conquer strategy \citep{LiGuZh2017}.

\begin{definition}[Distributed filtered hyperinterpolation]\label{defn:dfh}
Let $D:={(\bx_i,y_i)}_{i=1}^N$ be a data set on manifold $\mfd$.
The distributed filtered hyperinterpolation (DFH) for distributed data sets $\{D_j\}_{j=1}^{m}$ of $D$ is a synthesized estimator of local estimators $V_{D_j,n}$, $j=1,2,\dots,m$, each of which is
the filtered hyperinterpolation for noisy data $D_j$:
\begin{equation}\label{eq:dfh}
   V_{D,n}^{(m)}(\bx):=V_{D,n}(\{D_j\}_{j=1}^{m};\bx) 
   := \sum_{j=1}^m \frac{|D_j|}{|D|}
   V_{D_j,n}(\bx),\quad \bx\in\mfd,
\end{equation}
where for $j=1,\dots,m$, the local estimator is a filtered hyperinterpolation on $D_j$:
\begin{equation*}
    V_{D_j,n}(\bx)= \sum_{\bx_i\in D_j} w_{i}y_i K_n(\bx,\bx_i).
\end{equation*}
For noiseless data sets $D^*=\{(\bx_i,f^*(\bx_i))\}_{i=1}^N$ and $D^*_j$ associated with the target function $f^*$, denote the distributed filtered hyperinterpolation by 
\begin{equation}\label{eq:dfh_clean}
   V_{D^*,n}^{(m)}(\bx)
   = \sum_{j=1}^m \frac{|D^*_j|}{|D^*|}
   V_{D^*_j,n}(\bx),\quad \bx\in\mfd.
\end{equation} 
\end{definition}
The synthesis here is a process when the local estimators
communicate to a central processor to produce the global estimator
$V_{D,n}^{(m)}$. The weight in the sum of \eqref{eq:dfh} for each local server is proportional to the amount of data used in the server. 

\paragraph{Quadrature rule for distributed learning} For $n\in\mathbb N$, suppose the quadrature rule $\{(w^{'}_{i},\bx_i)\}_{i=1}^{|D|}$ satisfies the condition of Lemma \ref{lem:random_quadr} for distribution $\nu$ and polynomials of degree $n$. Using in total $m$ servers, $m\geq2$, we let
\begin{equation}\label{eq:wi_dfh_ran}
    w_{i}=\left\{\begin{array}{ll} w^{'}_{i},& \displaystyle\mbox{if}\
        \sum_{i=1}^{|D|}|w^{'}_{i}|^2\leq 2/m,\\[1mm]
        0,&\mbox{otherwise},
        \end{array}\right.\quad\forall i=1,\dots,|D|.
\end{equation}
We denote $\{(w_{i},\bx_i)\}_{i=1}^{|D|}$ by $\QH^{(m)}$. In the distributed filtered hyperinterpolation, we need to use this modified quadrature rule $\{(w_{i},\bx_i)\}_{i=1}^{|D|}$ with weights in \eqref{eq:wi_dfh_ran} to achieve good approximation performance.

\subsection{Distributed filtered hyperinterpolation for clean data}\label{sec:dfh_clean}
The synthesis here is a process when the local estimators
communicate to a central processor to produce the global estimator
$V_{D,n}^{(m)}$. 
Like the non-distributed case, we start with the case of deterministic sampling. 
The following theorem shows that the distributed filtered hyperinterpolation $V_{D,n}^{(m)}$ has a similar
approximation performance as the non-distributed $V_{D,n}$ when the number of local servers is not too large as compared with the amount of data. 
\begin{theorem}[DFH for clean and deterministic data]\label{thm:dfh_clean_det}
Let $d\geq 2$ and $1\leq p\leq \infty$, $r>d/p$, $n\in\N$, $m\geq2$ and $D^*$ a clean data set satisfying \eqref{eq:noisydats}. Let $\{D^*_j\}_{j=1}^{m}$ be $m$
distributed data sets of $D^*$. Let $\fiH$ be a filter given by Definition~\ref{def:fiH} with $\fis\ge d+1$. For $j=1,\dots,m$, the data set $D^*_j$ on the $j$th server satisfies that $\QH[D_j^*]$ is a positive quadrature rule exact for polynomials of degree up to $3n-1$. Then, for $f^*\in\sobm{p}{s}$ with $s>d/p$,
\begin{equation*}
    \bigl\|V_{D^*,n}^{(m)}-f^*\bigr\|_{L_2(\mfd)}\leq C n^{-r},
\end{equation*}
where $C$ is a constant independent of $|D^*|$, $|D^*_1|,\dots,|D^*_m|$
and $n$.
\end{theorem}

Theorem~\ref{thm:dfh_clean_det}
illustrates that with the same assumption as Theorem~\ref{thm:nondfh_clean_det}, the distributed filtered hyperinterpolation has the same approximation performance as the non-distributed case, where the latter processes all the distributed data sets in a single server. 

The distributed filtered hyperinterpolation with random sampling is
a weighted average of individual non-distributed filtered hyperinterpolations on local servers, where each weight is in proportion to the amount of the data used by the corresponding local server. Let $V_{D^*_j,n}$ be the non-distributed filtered
hyperinterpolation for clean data $D^*_j$ with random sampling points. We define the global estimator $V_{D^*,n}^{(m)}$ as \eqref{eq:dfh}. 
In the following theorem, we show that the approximation error for $V_{D^*,n}^{(m)}$ on $d$-manifold converges at rate $|D^*|^{-r/d}$ where $r$ is the smoothness of the target function. 

\begin{theorem}[DFH for clean data and random samples]\label{thm:dfh_ran_clean}
 Let $d\geq 2$, $r>d/2$, $m\geq2$ and $D^*=\{(\bx_i,f(\bx_i))\}_{i=1}^{|D^*|}$ and its $m$ partition sets $D^*_j$, $j=1,\dots,m$. The sampling points are i.i.d.~random points on $\mfd$ with distribution $\mu$ in \eqref{eq:condition on distribution}. If $\min_{j=1,\dots,m}|D^*_j|\geq c n^{d+\tau}$ given $0<\tau\leq d$, and $|D^*|\leq c' n^{2d}$, for two positive constants $c,c'$, then for the target function $f^*\in \mathbb W_2^r(\mfd)$,
\begin{equation*}
    \mathbf{E}\left\{\|V_{D^*,n}^{(m)}-f^*\|_{L_2(\mfd)}^2\right\} \leq C |D^*|^{-r/d},
\end{equation*}
where $C$ is a constant independent of $|D^*|$, $|D^*_1|,\dots,|D^*_m|$ and $n$.
\end{theorem}
The proofs of Theorems~\ref{thm:dfh_clean_det} and \ref{thm:dfh_ran_clean} are deferred to Section~\ref{appendix:dfh_clean}. 

\subsection{Distributed filtered hyperinterpolation for noisy data}
\label{sec:dfh_noisy}
In this subsection, we describe the distributed learning by filtered hyperinterpolation for noisy data with deterministic and random sampling.
As shown in the following theorem, we prove that the distributed filtered hyperinterpolation $V_{D,n}^{(m)}$ has similar
approximation performance as the non-distributed $V_{D,n}$ if the number of local servers is not large or each server has a sufficient amount of data.
\begin{theorem}[DFH for noisy data and deterministic samples]\label{thm:dfh_noisy_det}
Let $d\geq 2$, $r>d/2$, $m\geq2$ and $D$ a noisy data set
satisfying \eqref{eq:noisydats}. Let $\{D_j\}_{j=1}^{m}$ be $m$
distributed data sets of $D$. For $j=1,\dots,m$, the sampling
point set $\Lambda_{D_j}$ of $D$ satisfies Assumption~\ref{assum:qrmfd}. If the distributed filtered hyperinterpolation $V_{D,n}^{(m)}$ for $\{D_j\}_{j=1}^{m}$ satisfies that the target function $f^*$ is in $\mathbb W_2^r(\mfd)$, $\frac{c_3}6|D|^{\frac{1}{2r+d}}\leq n\leq \frac{c_3}3|D|^{\frac{1}{2r+d}}$ for the constant $c_3$ in \eqref{eq:qrmfd}, and
$\min_{j=1,\dots,m}|D_j|\geq |D|^{\frac{d}{2r+d}}$, then,
\begin{equation*}
    \mathbf{E}\left\{\|V_{D,n}^{(m)}-f^*\|_{L_2(\mfd)}^2\right\}\leq C_2 |D|^{-2r/(2r+d)},
\end{equation*}
where $C_2$ is a constant independent of $|D|$, $|D_1|,\dots,|D_m|$
and $n$.
\end{theorem}

The distributed filtered hyperinterpolation for deterministic sampling has the same order $|D|^{-2r/(2r+d)}$ of the approximation error as compared to the non-distributed case in Theorem~\ref{thm:ndfh_noisy_det}. Thus, appropriately distributing data to local servers, the divide-and-conquer strategy does not reduce the approximation capability of filtered hyperinterpolation. We will see that it is also true when the sampling is random. 

\begin{remark}\label{rem:thm:dfh_noisy_det} Suppose each server takes the same number of data. With less than $|D|^{\frac{2r}{2r+d}}$ servers, the $L_2$ error for the product space $\Omega\times L_2(\mfd)$ converges at rate $|D|^{\frac{1}{1+d/(2r)}}$.
 The condition $\min_{j=1,\dots,m}|D_j|\geq
|D|^{\frac{d}{2r+d}}$ has a close connection to the number $m$ of local servers. In particular, if $|D_1|=\dots=|D_m|$, the condition $\min_{j=1,\dots,m}|D_j|\geq |D|^{\frac{d}{2r+d}}$ is equivalent with $m\leq |D|^\frac{r}{r+d/2}$.
\end{remark}

When the data $D$ is noisy with random sampling points, 
the distributed $V_{D,n}^{(m)}$ in \eqref{eq:dfh} has the same approximation rate as the non-distributed case in Theorem~\ref{thm:ndfh_noisy_ran}.
\begin{theorem}[DFH for noisy data and random samples]\label{thm:dfh_ran_noisy}
 Let $d\geq 2$, $r>d/2$, $m\geq2$ and $D$ a noisy data set
satisfying \eqref{eq:noisydats}. The sampling points are i.i.d.~random points on $\mfd$ with distribution $\mu$ in \eqref{eq:condition on distribution}. If the target
function $f^*\in \mathbb W_2^r(\mfd)$, $n\asymp |D|^{1/(2r+d)}$
and $\min_{j=1,\dots,m}|D_j|\geq |D|^{\frac{d+\tau}{2r+d}}$ for some $\tau$ in $(0,2r)$, then,
\begin{equation}\label{eq:errdfh.noise.random}
    \mathbf{E}\left\{\|V_{D,n}^{(m)}-f^*\|_{L_2(\mfd)}^2\right\} \leq C_4|D|^{-2r/(2r+d)},
\end{equation}
where $C_4$ is a constant independent of $|D|$, $|D_1|,\dots,|D_m|$ and $n$.
\end{theorem}

\begin{remark}\label{rem:thm:dfh_ran_noisy} Note that the approximation rate $|D|^{-2r/(2r+d)}$ in \eqref{eq:errdfh.noise.random} is the same as Theorem~\ref{thm:dfh_noisy_det} when the sampling points are deterministic. It means with appropriate random distribution of the sampling points, the randomness of sampling does not reduce the approximation performance of distributed filtered hyperinterpolation.
If $|D_1|=\dots=|D_m|$, the condition $\min_{j=1,\dots,m}|D_j|\geq |D|^{\frac{d+\tau}{2r+d}}$ is equivalent with $m\leq |D|^\frac{r-\tau/2}{r+d/2}$.
\end{remark}

We postpone the proofs of Theorems~\ref{thm:dfh_noisy_det} and \ref{thm:dfh_ran_noisy} to Section~\ref{appendix:dfh_noisy}.

\section{Examples and numerical evaluation} 
\label{sec:example}
We illustrate the notions and filtered hyperinterpolation for single and multiple servers on the 2-d mathematical torus $\torus{2}$. The torus $\torus{2}$ can be parameterized by the product of unit circles $\sph{1}\times \sph{1}$ and is equivalent to $[-\pi,\pi]^2$. Denote $L_2(\torus{2})$ the $L_2$ space on $\torus{2}$ with the Lebesgue measure. On the manifold $\torus{2}$, the Laplacian
 \begin{equation*}
     \LBm:=\frac{\partial^{2}}{\partial x_{1}^{2}} + \frac{\partial^{2}}{\partial x_{2}^{2}}
 \end{equation*} 
is the Laplace-Beltrami operator with eigenfunctions $\{\frac{1}{2\pi}\exp(\imu\: \bk\cdot \bx)\}_{\bk\in\Zd[2]}$ of $\bx\in \torus{2}$ and eigenvalues $\{|\bk|^{2}\}_{\bk\in\Zd[2]}$, where $\imu:=\sqrt{-1}$ is the imaginary unit, $\bk\cdot \bx=k_{1}x_{1}+k_2 x_2$ and $|\bk|:=\sqrt{k_1^2+k_2^2}$. 
Here $\bk=(k_1,k_2)$ and $\bx=(x_1,x_2)$. The space of polynomials of degree $n$ is $\Pi_n:={\rm span}\{\frac{1}{2\pi} e^{\imu\: \bk\cdot \bx}: |\bk|\leq n\}$. 
For $1\le p\le \infty$, let $\Lpt{p}$ be the $\mathbb{L}_{p}$ space with respect to the normalized Lebesgue measure $\IntDd{x}$ on $\torus{2}$. 

For our illustration, we define the filter $\fiH$ by the piece-wise polynomial function with $H(t)=1$ for $0\leq t\leq 1$; 
\begin{align}
    H(t) &= 1 + (t-1)^6\bigl[-462 + 1980(t-1) - 3465(t-1)^2 + 3080(t-1)^3 \notag\\
     &\qquad - 1386(t-1)^4 + 252(t-1)^5\bigr] 
    \label{eq:filter}
\end{align}
for $t\in (1,2)$; and $H(t)=0$ for $t\geq 2$. 
Then $\fiH$ is in $C^5(\Rplus)$ and satisfies Definition~\ref{def:fiH}. Figure~\ref{fig:fiH_C5_wend} shows the plot of this filter. 
This particular filter has been used in previous works for the sphere, see \cite{SlWo2012,WaLeSlWo2017,Wang2016,WaSlWo2018}. 
We observe that the filter, which is constant $1$ over $[0,1]$, enables the filtered approximation and filtered hyinterpolation of degree $n$ (as given below) to reproduce polynomials with degree up to $n$ on $\torus{2}$. 
The finite support $[0,2]$ of $H$ makes the filtered hyperinterpolation a polynomial of degree up to $2n-1$. 
The middle polynomial over $[1,2]$ which is sufficiently smooth at the two ends modifies the Fourier coefficients from degree $n+1$ to $2n-1$ and makes the resulting filtered hyperinterpolation a near best approximator, as shown by Theorem~\ref{thm:nondfh_clean_det}.
\begin{figure*}
\centering
\begin{tabular}{cc}
{\tiny \sffamily Filter} & {\tiny \sffamily Wendland-Wu function}\\
\includegraphics[height=.3\textwidth]{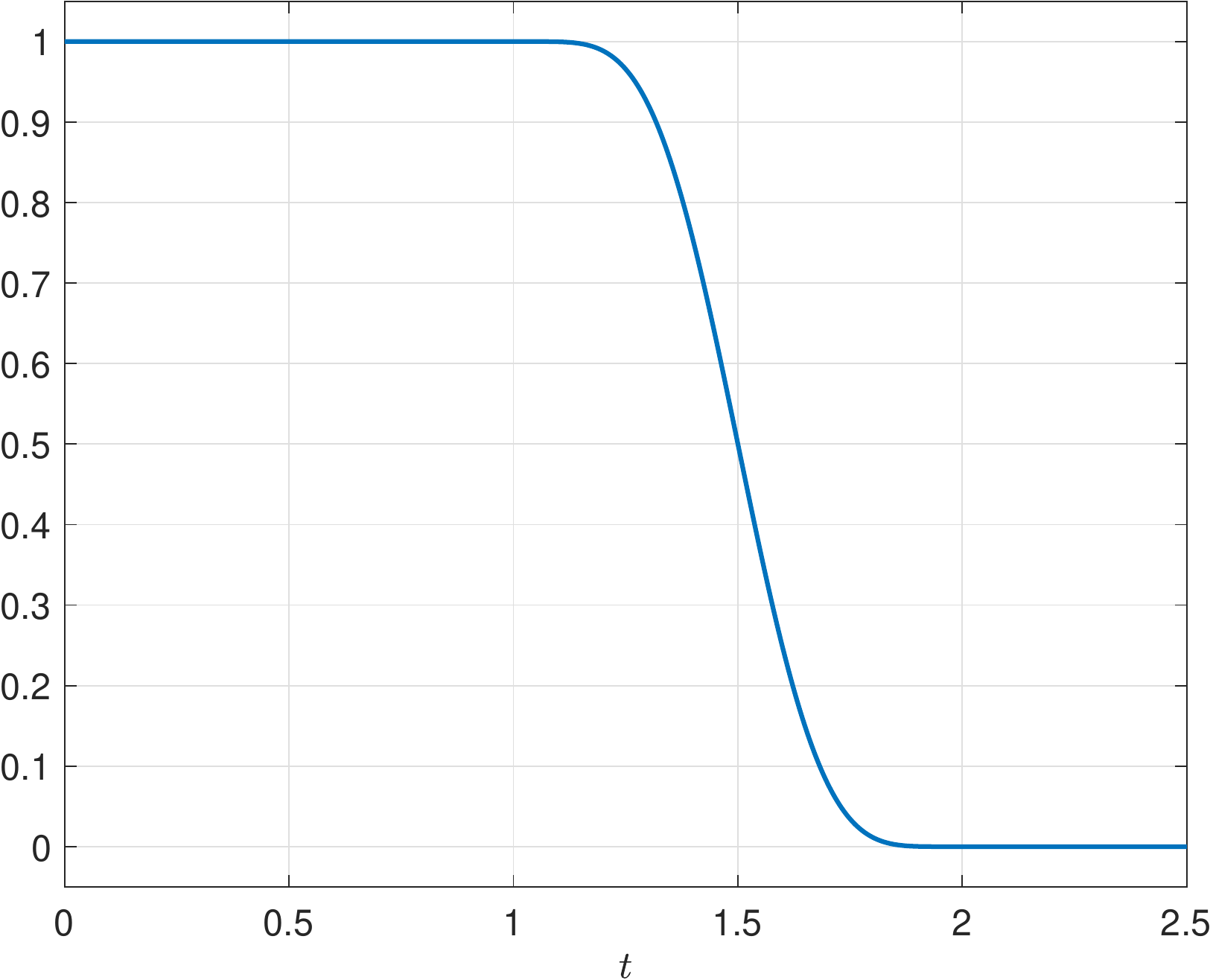}& \hspace{2mm}
\includegraphics[clip=true,trim={0 -8mm 0 8mm}, height=0.3\textwidth]{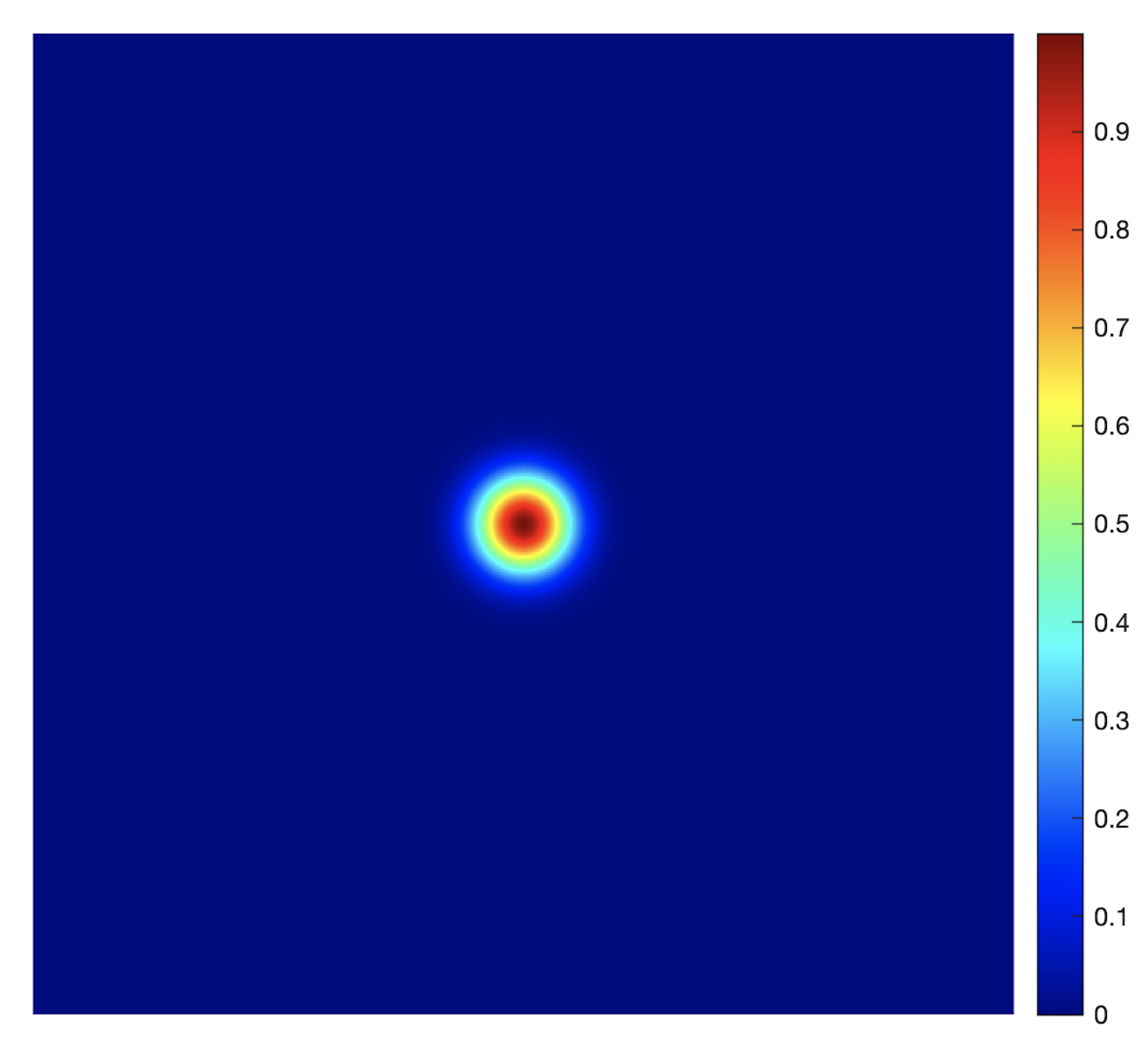}
\end{tabular}
\caption{Left: The filter $H$ in $C^5(\Rplus)$ given in \eqref{eq:filter}. Right: Wendland-Wu function on the torus, heatmap.}\label{fig:fiH_C5_wend}
\end{figure*}
With the filter \eqref{eq:filter}, the filtered kernel on $\torus{2}$ with filter $H$ is, for $n\in\N$ and $\bx,\by\in \torus{2}$,
\begin{equation*}
   \vdh{n,\fiH}(\bx,\by):= \frac{1}{(2\pi)^2}\sum_{\bk\in\Zd[2]}\fiH\left(\frac{|\bk|}{n}\right)e^{\imu\bk\cdot(\bx-\by)}.
\end{equation*}
As the support of filter $H$ is $[0,2]$, the summation over $\bk$ is constrained to $|\bk|\leq 2n-1$. 
The filtered approximation for $f\in L_2(\torus{2})$ is then 
\begin{equation*}
   V_{n,\fiH}(f;\bx):= \int_{\torus{2}} K_{n,\fiH}(\bx,\by)f(\by)\IntDd{y}.
\end{equation*}
As corresponds to Definition~\ref{defn:nondfh_det}, this is the ideal approximation of degree $n$, which is hard to compute as it would require integrating the unknown target function.

To construct a non-distributed filtered hyperinterpolation on $\torus{2}$, we 
consider $N=9n_0^2$ points $\bx_{j,l}=(2j\pi/(3n_0),2l\pi/(3n_0))$. 
For these we can use the quadrature rule $\QH=\{(w_{j,l},\bx_{j,l}): j,l=0,1,\dots,3n_0-1\}$ with 
$N=9n_0^2$ equal weights $w_{j,l}\equiv (2\pi)^2/N$. 
The quadrature rule $\QH$ is exact for polynomials of degree $n$. 
To satisfy the conditions of Theorems~\ref{thm:ndfh_noisy_det} and \ref{thm:dfh_noisy_det}, we can let hyperparameter $n_0=n$. 
In general, the quadrature weights need to be constructed depending on the location of the input points $\bx_{j,l}$. 
We have been able to show the existence of such weights for randomly sampled inputs (see Lemma~\ref{lem:random_quadr}); however, the explicit construction in such cases is yet to be explored. 

Consider a noisy data set $D=\{(\bx_{j,l},y_{j,l}): j,l=0,1,\dots,n_0-1\}$ with $y_{j,l} = f^*(\bx_{j,l}) + \epsilon_{j,l}$, $j,l=0,1,\dots,n_0 -1$, and $f^*\in C(\mfd)$. 
Here $f^*$ is the ideal (noiseless) target function.
The non-distributed filtered hyperinterpolation of degree $n$ with filter $H$ and quadrature rule $\QH$ for data $D$ is given by
\begin{equation}\label{eq:det_ndfhT2}
    V_{D,n}(\bx) = \frac{1}{N}\sum_{j,l=0,1,\dots,3n_0-1} y_{j,l} \sum_{|\bk|\leq 2n-1}H\left(\frac{|\bk|}{n}\right) e^{\imu (\bx-\bx_{j,l})\cdot \bk}, \quad \bx\in \torus{2}.
\end{equation}
This is our construction to obtain an approximation. It corresponds to Definition~\ref{defn:nondfh_det}. The summation index $\bk$ runs over a ball of radius $2n-1$ due to the compact support of the filter $H$.
Thus, $V_{D,n}(\bx)$ is fully discrete and computable. 
By Theorem~\ref{thm:ndfh_noisy_det}, the approximation error of $V_{D,n}$ for $f^*\in\mathbb{W}^r_2(\torus{2})$, $r>1$, has convergence rate at least of order $|D|^{-r/(r+1)}$ as $n_0$ (controlling the number of data points) and $n$ (controlling the degree of the approximation) increase.
In particular, if $f^*$ is a basis element in $\{\frac{1}{2\pi}e^{\imu \bx\cdot \bk}\}$,  then $r=\infty$, since a polynomial is infinitely smooth. 

In practice, we use the real part as the approximation and discard the complex part. Note that the amount of data, $9n^2$, determines the degree of the polynomials. In this example, the diffusion polynomials of degree $n$ are given by sums of eigenfunctions with momentum vector $\bk=(k_1,k_2)$ in the same grid defining the data.

\begin{figure}
\centering
    \scalebox{1}{ 
    \begin{tikzpicture}
    \node at (0,0) {\includegraphics[clip=true,trim=2cm 0.5cm 3.95cm 0.5cm,scale=.55]{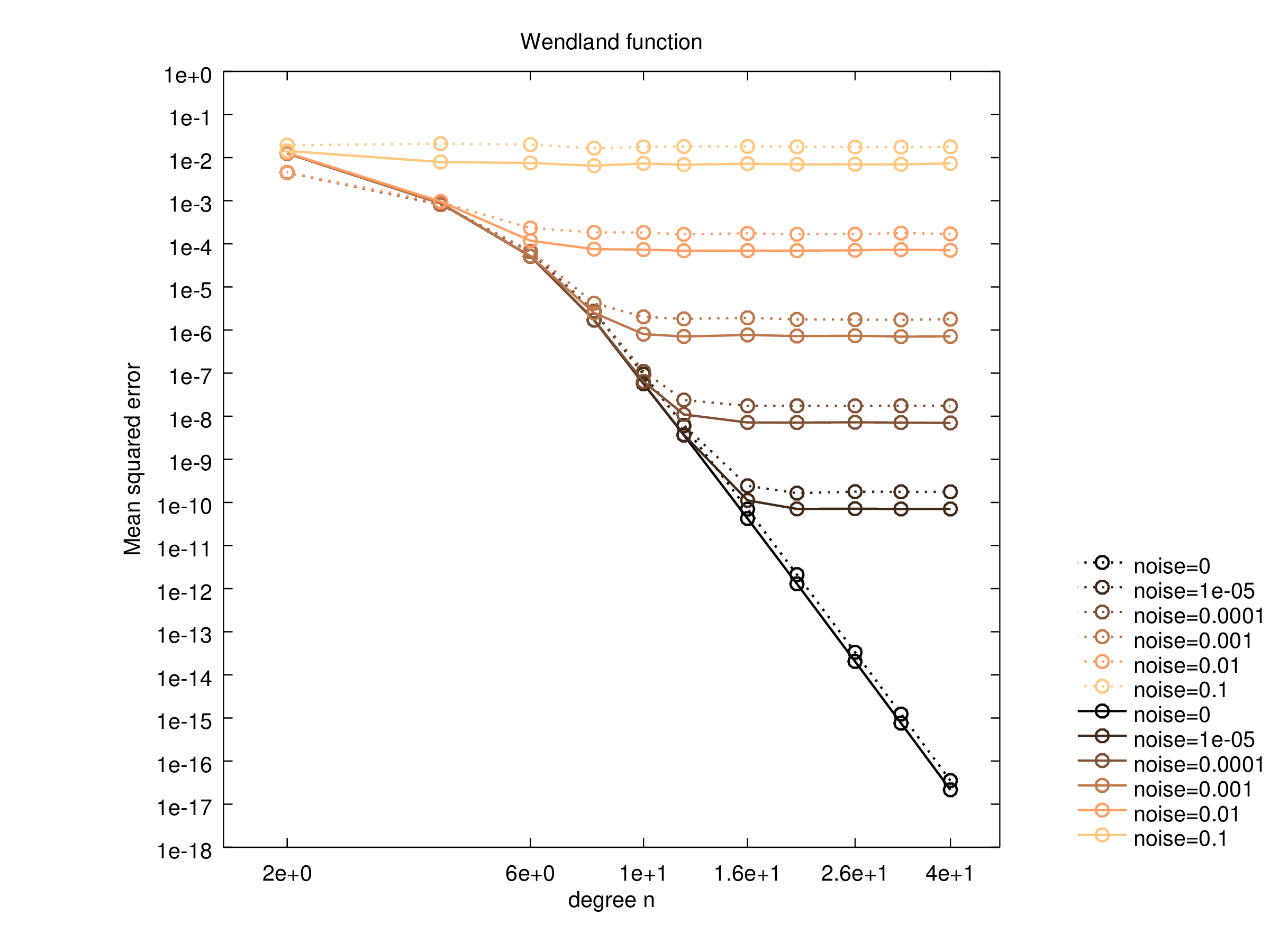}};
    \node (T1) at (-1.5,-1.5) {\tiny \sffamily train};
    \node (T2) [right of = T1, node distance = .75cm] {\tiny \sffamily general};
    \node (L1) [below of = T1, node distance = .8cm] {\includegraphics[clip=true,trim=17cm 4.06cm 2.2cm 8.84cm,scale=.5]{Err-target0}};
    \node (L2) [right of = L1, node distance =.75cm] {\includegraphics[clip=true,trim=17.1cm 1.65cm 2.25cm 11.2cm,scale=.5]{Err-target0}};
\node (L3) [right of = L2, node distance =1.125cm] {\includegraphics[clip=true,trim=18.1cm 1.6cm 0cm 11.25cm,scale=.5]{Err-target0}};

\node (I) at (6.5,4) {}; 
\node (T) [below of = I, node distance = .2cm] {\tiny \sffamily \begin{tabular}{cc}
data (noise=0.01) & trained function\end{tabular}}; 
\foreach \x [evaluate=\x as \jx using {int(9*\x*\x)}] in {2, 4, 6, 8, 10, 12, 16}
{ 
\node (I) [below of = I, node distance = 1cm] {\includegraphics[clip=true,trim=7cm 5.5cm 1cm 5.25cm, scale=.25]{target0-n\x-noise5}}; 
\node (J) [right of = I, node distance = 1.5cm] {\tiny \sffamily \x }; 
\node (K) [left of = I, node distance = 1.5cm] {\tiny \sffamily \jx }; 
}
\end{tikzpicture}
}
    \caption{$L_2$ squared errors of non-distributed learning by filtered hyperinterpolation for data from the Wendland-Wu function \eqref{eq:wendland-wu} on the torus $\torus{2}$, for different levels of noise in the training data, as the degree $n$ for the approximation (and the sample size $N=(3n)^2$) increases. 
    In each case, the function is learned using the noisy training data. Dotted lines show the error on the training data (training error), and solid lines show the population error relative to the ideal function (generalization error). The right part shows a few examples of the noisy training data and the corresponding learned functions, alongside with the number of data points and the degree. 
    The result is very stable, over repetitions.}
    \label{fig:non-distributed-Wendland}
\end{figure}

\begin{figure}
\centering
    \scalebox{1}{ 
    \begin{tikzpicture}
    \node at (0,0) {\includegraphics[clip=true,trim=1.5cm 0.5cm 3.95cm 0.5cm,scale=.54]{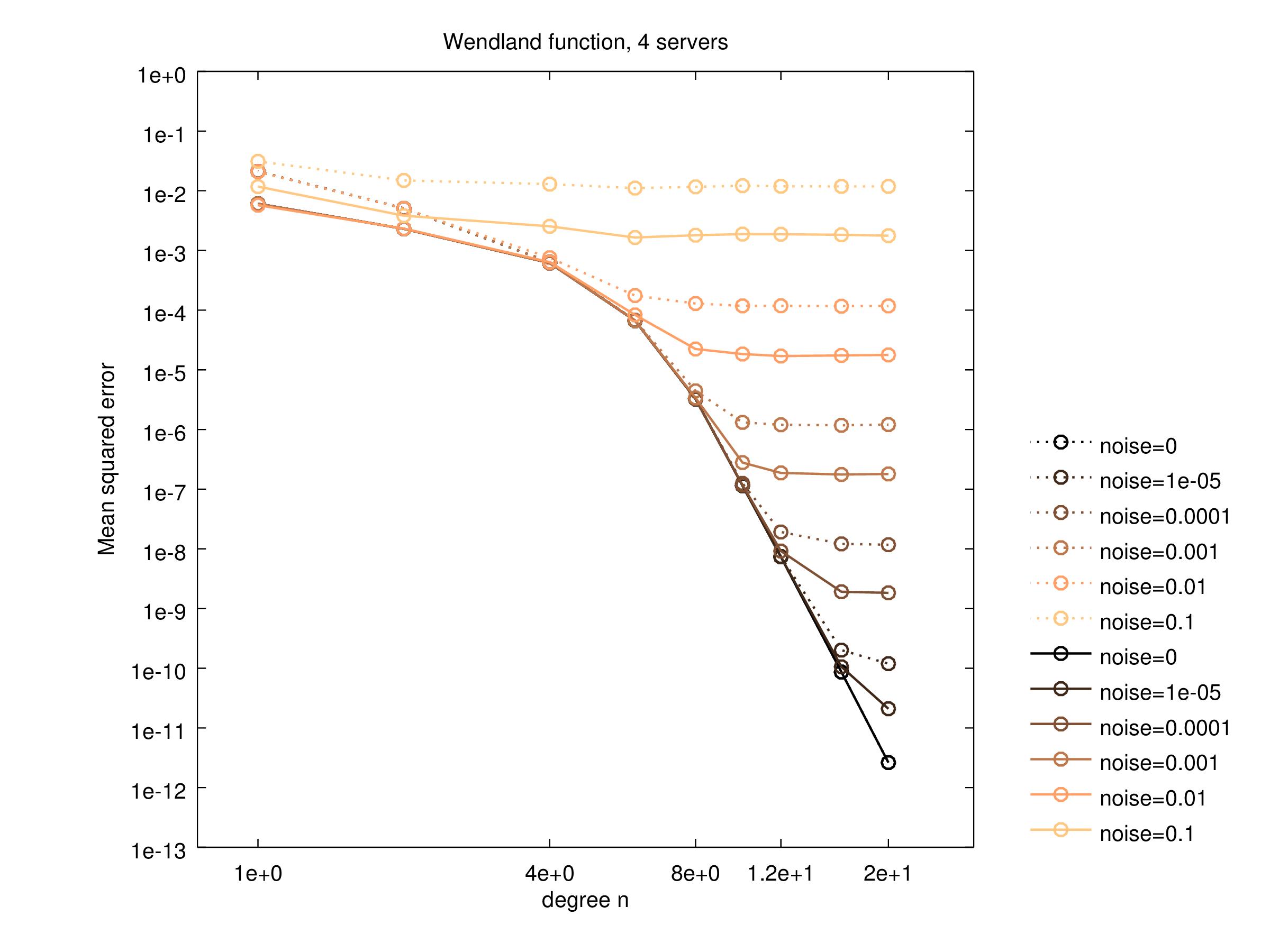}}; 
    \node (T1) at (-1.5,-1.5) {\tiny \sffamily train};
    \node (T2) [right of = T1, node distance = .75cm] {\tiny \sffamily general};
    \node (L1) [below of = T1, node distance = .8cm] {\includegraphics[clip=true,trim=17cm 4.06cm 2.2cm 8.84cm,scale=.5]{Err-target0}}; 
    \node (L2) [right of = L1, node distance =.75cm] {\includegraphics[clip=true,trim=17.1cm 1.65cm 2.25cm 11.2cm,scale=.5]{Err-target0}}; 
\node (L3) [right of = L2, node distance =1.125cm] {\includegraphics[clip=true,trim=18.1cm 1.6cm 0cm 11.25cm,scale=.5]{Err-target0}}; 

\node (I) at (6.4,4) {}; 
\node (T) [below of = I, node distance = .2cm] {\tiny \sffamily \begin{tabular}{cc}
data (noise=0.01) & trained function\end{tabular}}; 
\foreach \x [evaluate=\x as \jx using {int(9*\x*\x*4)}] in {1,2, 4, 6, 8, 10, 12}
{ 
\node (I) [below of = I, node distance = 1cm] {\includegraphics[clip=true,trim=7cm 5.5cm 1cm 5.25cm, scale=.25]{target0-n\x-noise5-servers4}
}; 
\node (J) [right of = I, node distance = 1.5cm] {\tiny \sffamily \x }; 
\node (K) [left of = I, node distance = 1.5cm] {\tiny \sffamily \jx }; 
}
\end{tikzpicture}
}
    \caption{$L_2$ squared error of distributed learning by filtered hyperinterpolation with $m=4$ servers for data from the Wendland-Wu function \eqref{eq:wendland-wu} on the torus $\torus{2}$, for various levels of noise, as the degree $n$ for the approximation (and the total sample size $N=(3n)^2\cdot m$) increases. 
    Dotted lines show the error on the training data, and solid lines show the population error relative to the ideal function. 
    The right part shows examples of noisy training data and the corresponding learned functions. The training data is split into 4 interleaved pieces for processing, and the final trained function is the average of the functions obtained in the local servers.
    }
    \label{fig:distributed-Wendland}
\end{figure}

The corresponding distributed filtered hyperinterpolation with $m$ servers is 
\begin{equation*}
    V_{D,n}^{(m)}(\bx) = \sum_{j=1}^{m} \frac{|D_j|}{|D|} V_{D_j,n}(\bx),\quad \bx\in \torus{2},
\end{equation*}
where $D_j$ is the data set on the $j$th local server, for $j=1,\dots,m$. 
This corresponds to Definition~\ref{defn:dfh}. 
By Theorem~\ref{thm:dfh_noisy_det}, the approximation error of the distributed strategy $V_{D,n}^{(m)}$ with $m$ servers for $f^*\in \mathbb{W}_2^r(\torus{2})$, $r>1$, is at least of order $|D|^{-r/(r+1)}$ provided the number of servers used satisfies $m\leq |D|^{\frac{r}{r+d/2}}$. 

When the points of data set $D$ are randomly distributed and satisfy the condition of Lemma~\ref{lem:random_quadr}, the non-distributed filtered hyperinterpolation remains the same as \eqref{eq:det_ndfhT2} with the points $\bx_{j,l}$ replaced by the set of random points $\Lambda_{D}$. But the local estimator $V_{D_j,n}$ in the distributed filtered hyperinterpolation uses the modified weights
\begin{equation*}
    w^*_{j,l}:=\left\{
    \begin{array}{ll} (2\pi)^2/N,& \displaystyle\mbox{if}\
        \sum_{0\leq k,l\leq n-1}|w_{j,l}|^2\leq 2/m,\\[4mm]
        0,&\mbox{otherwise}, 
        \end{array}\right. 
\end{equation*}
in the place of \eqref{eq:det_ndfhT2}.

For our illustration, we use the Wendland-Wu function on the torus as the target function:
\begin{equation}\label{eq:wendland-wu}
    f(\bx) = \phi(\bx-\bx_c),\quad \bx\in \torus{2}, 
\end{equation}
where $\phi(u)$ is the one-dimensional Wendland-Wu function
\begin{equation*}
  \phi(u) := 
  (1-u)_{+}^{8}(32u^3 + 25u^2 + 8u + 1),
\end{equation*}
and $\bx_c=(0,0)$ is the center, see \cite{Wendland1995piecewise,Wu1995compactly}. 
We show in the right part of Figure~\ref{fig:fiH_C5_wend} the Wendland-Wu function in \eqref{eq:wendland-wu} which is in $C^{6}(\torus{2})$. 
We generate noisy data set by adding Gaussian white noise at a particular noise level to the values from the Wendland-Wu function. 

Figure~\ref{fig:non-distributed-Wendland} shows the $L_2$ squared errors of both training and generalization for the approximation by non-distributed filtered hyperinterpolation on noisy data from the Wendland-Wu function \eqref{eq:wendland-wu} on $\torus{2}$, with six levels of noise from 0 to 0.1. 
The degree $n$ for the approximation is up to 40, and the sample size is $N=(3n)^2$. 
The right part shows a few examples of the noisy training data, all at a noise level of $0.01$, and the corresponding learned functions. 
For noiseless case, the training and generalization errors both converge to zero rapidly at a rate of approximately $\|V_{D,n}-f^*\|_{L_2(\mathcal{M})}\sim N^{-4}$. This is consistent with the theoretical upper bound $N^{-3}$ given in Theorem~\ref{thm:nondfh_clean_det} where $s\geq6$ and $d=2$. The slightly higher rate $N^{-4}$ is due to that the $\phi(\bx)$ may have a higher smoothness. 
For noisy data, the convergence of the error stops at a particular degree. 
The convergence rate is higher when the noise level is smaller. 
The mean squared error on the training data converges to a value close to the square of the noise level, which indicates that the trained function is filtering out the noise. 
For both noisy and noiseless cases, the generalization error is slightly lower than the training error. 
Also, the result has consistent stability over repetitions in all cases.

Figure~\ref{fig:distributed-Wendland} shows the $L_2$ squared errors for distributed filtered hyperinterpolation. We also generate the data from Wendland-Wu function. For this experiment, we partition the data set equally into $m=4$ servers. 
The $i$th server computes a filtered hyperinterpolation on the data $D_i$ which are defined on interleaved grids of the form ${\rm mod}(\bx_{j,l}+ \mathbf{s}_i,2\pi)$, where $\mathbf{s}_i$ is a shift number between $(0,2\pi)$ and $\mathbf{s}_i$ are distinct for different subsets $D_i$. 
The quadrature rule $\QH[D_i]$ utilizes equal weights as the non-distributed case.
The distributed filtered hyperinterpolation combines the results from all servers, which has similar approximation behaviour as the non-distributed case. If using noisy data in training, the approximation error has saturation after a particular degree; while with noiseless data, the error decays to zero all through the degree. We observe here that the generalization error has a more significant gap with training error as compared to the non-distributed case, which may be partly due to the distributed strategy (on multiple servers) are adopted. These experiments show consistent results as the theory in previous sections.

\section{Discussion}
\label{sec:comparison}
\paragraph{Rates of convergence}
In Table~\ref{tab:comparison}, we compare the theoretical convergence rates of the non-distributed and distributed filtered hyperinterpolation in noiseless and noisy cases, as obtained in the previous sections. It shows that the filtered hyperinterpolation for clean data can achieve an optimal convergence rate $N^{-r/d}$ in both non-distributed and distributed cases and both deterministic and random sampling cases. 
For noisy data, the non-distributed filtered hyperinterpolation has a slightly lower approximation rate at $N^{-r/(r+d/2)}$, $r>d/2$, which in the limiting case $r\to d/2$ becomes the optimal rate $N^{-r/d}$. 
The distributed strategy preserves the convergence rate $N^{-r/(r+d/2)}$ of the non-distributed filtered hyperinterpolation for noisy data, provided that the number of data $N$ increases sufficiently fast with the number of servers, but the condition of the number of servers in the deterministic sampling case is weaker than the random sampling case.

\begin{table}[t]
\begin{minipage}{\textwidth}
\caption{Behavior of the error upper bound in the four settings that we considered in the paper, 
depending on the number of data points $N=|D|$, the smoothness $r$ of the target function, the manifold dimension $d$, and the number of servers $m$. In all cases, the noise in the data lowers the rate order of the approximation error. In noisy cases, the constant contains two terms: one is a constant times the squared Sobolev norm of the target function; the other is a constant times the squared noise upper bound. 
}
\label{tab:comparison}
\begin{center}
\begin{tabular}{P{2.5cm}|P{2.5cm}p{2.5cm}|P{2.5cm}P{2.5cm}}
\toprule
 \multirow{2}{*}{Type} & \multicolumn{2}{c}{clean} & \multicolumn{2}{c}{noisy}\\ \cline{2-3} \cline{4-5}
 & deterministic & random & deterministic & random\\ 
\midrule
  Non-distributed 
  & $N^{-r/d}$ \newline Th.~\ref{thm:nondfh_clean_det} 
  & $N^{-r/d}$ \newline Th.~\ref{thm:ndfh_ran_clean}
  & $N^{-r/(r+d/2)}$ \newline Th.~\ref{thm:ndfh_noisy_det} & $N^{-r/(r+d/2)}$ \newline Th.~\ref{thm:ndfh_noisy_ran} \\
\midrule
  Distributed ($m$ servers) 
    & $N^{-r/d}$ \newline ($m\leq N$) \newline Th.~\ref{thm:dfh_clean_det}  
    & $N^{-r/d}$ \newline ($m\leq \frac{N}{n^{d+\tau}}$) \newline Th.~\ref{thm:dfh_ran_clean}
    & $N^{-r/(r+d/2)}$ \newline ($m\leq N^{\frac{r}{r+d/2}}$) \newline Th.~\ref{thm:dfh_noisy_det}
  & $N^{-r/(r+d/2)}$ \newline 
    ($m\leq N^{\frac{r-\tau/2}{r+d/2}}$)
    \newline Th.~\ref{thm:dfh_ran_noisy}
    \\
\bottomrule
\end{tabular}
\end{center}
\end{minipage}
\end{table}

\paragraph{Implementation and complexity}
We already illustrated the computation of the method in Section~\ref{sec:example}. A summary of the implementation is shown in Algorithm~\ref{algorithm}\footnote{The condition $m\leq \sqrt{N}$ in Algorithm~\ref{algorithm} is a consequence of Remarks~\ref{rem:thm:dfh_noisy_det} and \ref{rem:thm:dfh_ran_noisy}.}. 
In deterministic sampling, we start with some given input data and a suitable quadrature rule for those input values. In the random sampling case, we begin with the data, which in theory is only assumed sampled from some distribution, and then construct a suitable quadrature rule. 
There are various details to consider for the implementation. 
First, we need to choose a filter $H$, which should be sufficiently smooth depending on the dimension of the manifold $\mathcal{M}$. The support of the filter will constrain the degree of the polynomials in the approximation. 
Second, we need a quadrature rule. Once a quadrature rule has been determined for the input data on the manifold, it can be applied to any output data. 
For important families of manifolds and configurations of points, quadrature rules are available from the literature. For instance, on the torus, cubes \citep{trefethen2013approximation,driscoll2014chebfun}; sphere: Gaussian, spherical design \citep{HeSlWo2010,BoRaVi2013,DeGoSe1977,Womersley2018}; graph: its nodes. 
The practical computation of quadrature rules for general types of data (or random input data) is an interesting problem in its own right, which has yet to be developed in more detail. 
Once a quadrature rule is available, the time complexity for Algorithm~\ref{algorithm} is $\bigo{}{\max_{j=1,\dots,m} |D|^{\frac{d}{2r+d}} |D_j|}$. 
If $|D_j|$ are all equal, the time complexity becomes $\bigo{}{|D|^{\frac{d}{2r+d}+1}/m}$.

\begin{algorithm}[t]
    \KwIn{For a given $N$, a number $m \leq \sqrt{N}$ of servers; 
    a filter $H\colon \Rplus\to\mathbb{R}$; 
    a choice of the polynomial degree $n$ for all servers;
    If deterministic samples: for each $j=1,\ldots,m$, a quadrature rule $Q_j=\{(w_i^{(j)},\bx_i^{(j)})\}_{i=1}^{N/m}$ on a $d$-manifold $\mathcal{M}$ satisfying Assumption~\eqref{assum:qrmfd}; 
    data sets $D_j =\{(\bx_i^{(j)},y_i^{(j)})\}_{i=1}^{N/m}$, $j=1,\ldots,m$, of noisy samples $y_i^{(j)}$ of the outputs of a function $f^*$ at the input points $\bx_i^{(j)}$.
    }
  \KwOut{Approximation of function $f^*$ on $\mathcal{M}$ by distributed filtered hyperinterpolation of degree $n$} 
    \Begin{
    Identify the eigenvalues $\{\eigvm\}_{\ell\in\N}$ and eigenfunctions $\{\eigfm\in\Lpm{2}\: |\: \LBm\eigfm = -\eigvm^{2}\: \eigfm,\; \ell\in\N\}$ of the Laplace-Beltrami $\LBm$ on $\mathcal{M}$.  \\
    If random samples: Identify a suitable quadrature rule $Q_j=\{(w_i,\bx_i)\}_{i=1}^{|D_j|}$ for the data set $D_j$ on $\mathcal{M}$, satisfying \eqref{eq:wi_dfh_ran}.
    }
    \For{each server $j=1,\ldots,m$ in parallel}
    {Consider the filtered kernel functions $K_{n,H}(\cdot,\bx_i) = \sum_{\ell} H(\frac{\lambda_\ell}{n})\phi_\ell(\cdot)\overline{\phi_\ell(\bx_i)}$, $\bx_i\in D_j$, with filter $H$ and degree specification $n$ as shown in Def.~\ref{def:fil.fil.ker}. \\
    Compute the approximation $V_{D_j,n}(\cdot) = \sum_{i=1}^{|D_j|}w_i y_i K_{n,H}(\cdot,\bx_i)$ using the output data $y_i\in D_j$ and the quadrature rule $w_i\in Q_j$ as shown in Def.~\ref{defn:nondfh_det}.\\
    }
    \KwRet{The average $V_{D,n}^{(m)}(\cdot)=\sum_j\frac{|D_j|}{|D|}V_{D_j,n}(\cdot)$ as shown in Def.~\ref{defn:dfh}}.        
    \caption{Distributed filtered hyperinterpolation}
    \label{algorithm} 
\end{algorithm}

\paragraph{Final remarks}
We have provided the first complete theoretical foundation for distributed learning on manifolds by filtered hyperinterpolation. 
One appealing aspect of filtered hyperinterpolation is that it comes with strong theoretical guarantees on the error, which apply to the population error or generalization error. 
Obtaining accurate bounds of this kind with neural networks is an active topic of research (which needs to incorporate not only the theoretical capacity of the neural network but also implicit regularization effects from the parameter initialization and optimization procedures). 
In filtered hyperinterpolation, once the data and the corresponding approximation degree are given, the approximating function is computed in closed form, meaning that we do not require parameter optimization. 
Also, filtered hyperinterpolation is a method that allows us to tune the model complexity directly in terms of the amount of available data in a principled way. 
As we observe in numerical experiments, the population error often is better than the training error. 
An interpretation is that this method imposes priors in terms of the polynomial degree and thus it is able to filter out noise. 
The method incorporates the geometry of the input space through the basis functions which are utilized to construct the approximations. Here, the basis functions are eigenfunctions of the Laplace-Beltrami operator on the manifold. 
It also contributes to the interpretability of the approximations, which live in polynomial spaces for which we have a good intuition. 
On the downside, to obtain the approximating function, the method relies on numerical integration techniques, in particular, quadrature rules, which is non-trivial in general to obtain. 
For general Riemannian manifolds, we can use the eigenvalues and eigenvectors of the discrete version of the Laplacian to approximate the Laplace-Beltrami operator, where the sampling points can estimate the discrete Laplacian, see, e.g.\ \cite{Sunada2008,crane2013digital,dunson2019diffusion}.

\medskip
\paragraph{Acknowledgements}
Guido Mont\'{u}far and Yu Guang Wang acknowledge the support of funding from the European Research Council (ERC) under the European Union's Horizon 2020 research and innovation programme (grant agreement n\textsuperscript{o} 757983).
Yu Guang Wang also acknowledges support from the Australian Research Council under Discovery Project DP180100506.  
This material is based upon work supported by the National Science Foundation under Grant No.~DMS-1439786 while the authors were in residence at the Institute for Computational and Experimental Research in Mathematics in Providence, RI, during the Collaborate@ICERM on ``Geometry of Data and Networks''. 
Part of this research was performed while the authors were at the Institute for Pure and Applied Mathematics (IPAM), which is supported by the National Science Foundation (Grant No.~DMS-1440415).

\bibliographystyle{abbrvbib}
\bibliography{dlfh}   

\appendix
\newpage
\section{Proofs}\label{sec:proofs}
The appendices contain the proofs of the theorems in Sections \ref{sec:ndfh_clean}, \ref{sec:ndfh_noisy}, \ref{sec:dfh_clean} and \ref{sec:dfh_noisy} in turn.

\subsection{Proofs for Section~\ref{sec:ndfh_clean}}\label{appendix:ndfh_clean}
\begin{proof}[Lemma~\ref{lm:fihyper.reproduce.p}] Let $P\in\polyspm$ and $\bx\in\mfd$. By $\supp \fiH\subset[0,2]$ and Assumption~\ref{assump:poly}, $\vdh{n,\fiH}(\bx,\cdot)P(\cdot)$, for each $i=1,\dots, N$, is a polynomial of degree $3n-1$. 
Since $\fiH(t)=1$ for $t\in[0,1]$, and since $P$ and $\eigfm$, $\eigvm\ge n+1$, are orthogonal, then for $\bx\in\mfd$,
\begin{align}\label{eq:fiapprox.reproduce.poly}
    \fiapprox{n,\fiH}(P;\bx)
    &= \int_{\mfd} \sum_{\eigvm\le 2n}\fiH\Bigl(\frac{\eigvm}{n}\Bigr)\:\eigfm(\bx)\conj{\eigfm}(\PT{z}) P(\PT{z}) \dmf{z}\nonumber\\
    &= \int_{\mfd} \sum_{\eigvm\le n}\eigfm(\bx)\conj{\eigfm}(\PT{z}) P(\PT{z}) \dmf{z}
    = P(\bx).
\end{align}
The exactness of $\QH$ for degree $3n-1$ with \eqref{eq:fiapprox.reproduce.poly} then gives
\begin{align*}
    \fihyper{D,n}(P;\bx) &= \sum_{i=1}^{N} \wH\: \vdh{n,\fiH}(\bx,\pH{i}) P(\pH{i})\\
    &= \int_{\mfd}\vdh{n,\fiH}(\bx,\by)P(\by) \dmf{y} = \fiapprox{n,\fiH}(P;\bx) = P(\bx),
\end{align*}
thus completing the proof.
\end{proof}

\begin{proof}[Theorem~\ref{thm:fiapprox.err.Lp}] Let $P\in\polyspm$. 
By the linearity of $\fiapprox{n,\fiH}$ and Lemma~\ref{lm:fihyper.reproduce.p},
\begin{align*}
    \normb{f-\fiapprox{n,\fiH}(f)}{\Lpm{p}} &\le \norm{f-P}{\Lpm{p}} + \normb{\fiapprox{n,\fiH}(f-P)}{\Lpm{p}}\\
    &\le \left(1+\norm{\fiapprox{n,\fiH}}{\ptop}\right) \norm{f-P}{\Lpm{p}},
\end{align*}
which, as $P$ is an arbitrary polynomial in $\polyspm$, together with Theorem~\ref{thm:fiapprox.Bd} gives
\begin{equation*}
    \normb{f-\fiapprox{n,\fiH}(f)}{\Lpm{p}} \le c_{d,\fiH,\fis}\: \bestapprox{n}{f},
\end{equation*}
thus completing the proof.
\end{proof}

We go to prove Theorem~\ref{thm:nondfh_clean_det}, for which we need some lemmas as given below. 
The following theorem shows a Marcinkiewicz-Zygmund inequality for a quadrature rule on $\mfd$.
\begin{lemma}\label{lem:MZineq.mfd} Let $\QH=\{(\wH,\pH{i})\}_{i=1}^{N}$ be a positive quadrature rule on $\mfd$ satisfying for some $1\le p_{0}<\infty$, $c_{0}>0$ and $n\ge0$,
\begin{equation}\label{eq:MZineq.p0}
    \sum_{i=1}^{N} \wH | P(\pH{i})|^{p_{0}} \le c_{0} \int_{\mfd}|P(\by)|^{p_{0}}\dmf{y}, \quad P\in\polyspm.
\end{equation}
Then, for all $1\le p_{1}<\infty$ and $\ell>n$,
\begin{equation}\label{eq:MZineq.p1}
    \sum_{i=1}^{N} \wH | P(\pH{i})|^{p_{1}} \le c_{1} \left(\frac{\ell}{n}\right)^{d}\int_{\mfd}|P(\by)|^{p_{1}}\dmf{y},\quad P\in\polyspm[\ell],
\end{equation}
where $c_{1}$ depends only on $d$, $p_{0}$ and $c_{0}$.
\end{lemma}
\begin{remark} \cite{Dai2006} proved Lemma~\ref{lem:MZineq.mfd} when $\mfd$ is the unit sphere $\mfd$.
\end{remark}

The proof of Lemma~\ref{lem:MZineq.mfd} relies on the following lemma of \cite{FiMh2011}, which shows that the sum of the weights, the corresponding nodes of which lie in the region $\ball{\bx_{0},\beta,\beta+\alpha}$, is bounded by a constant multiple of the measure of this region.
\begin{lemma}\label{lm:R.property.MFD} Let $d\ge1$ and let $\mfd$ be a $d$-dimensional compact Riemannian manifold. Let $\QH:=\{(\wH,\pH{i})\}_{i=1}^{N}$ be a positive quadrature rule on $\mfd$ satisfying \eqref{eq:MZineq.p0} for some $1\le p_{0}<\infty$, $c_{0}>0$ and $n\in\Nz$. Then for $\beta\ge0$, $\alpha\ge1/n$ and $\bx_{0}\in\mfd$,
\begin{equation*}
    \sum_{\bx_{i}\in\ball{\bx_{0},\beta,\beta+\alpha}}\hspace{-2mm} \wH\le c\: \memf(\ball{\bx_{0},\beta,\beta+\alpha}),
\end{equation*}
where the constant $c$ depends only on $d$.
\end{lemma}

Let
\begin{equation}\label{eq:A.L}
    A_{\ell}(\theta):= \frac{\ell^{d}}{(1+ \ell \theta)^{d+1}},\quad \ell\in\N,\;\theta \in [0,\pi].
\end{equation}
Lemma~\ref{lm:R.property.MFD} implies the following estimate for a quadrature rule.
\begin{lemma}\label{lm:QH.loc.est} Let $d\ge1$ and let $\mfd$ be a $d$-dimensional compact Riemannian manifold. Let $\QH:=\{(\wH,\pH{i})\}_{i=1}^{N}$ be a quadrature rule on $\mfd$ satisfying \eqref{eq:MZineq.p0} for some $1\le p_{0}<\infty$, $c_{0}>0$ and $n\in\Nz$. Let $A_{n}(\theta)$ be given by \eqref{eq:A.L}. Then, for $\ell\ge n$,
\begin{equation*}
    \max_{\bx\in\mfd} \sum_{i=1}^{N}\wH\: A_{\ell}(\dist{\bx,\pH{i}}) \le c\: \left(\frac{\ell}{n}\right)^{d},
\end{equation*}
where the constant $c$ depends only on $d$.
\end{lemma}
\begin{proof} Let $\bx\in\mfd$. Since $\mfd$ is compact, $\mfd$ is bounded, i.e. there exists $0<r<\infty$ such that $\mfd\subseteq\ball{\bx,r}$. Using Lemma~\ref{lm:R.property.MFD},
\begin{align*}
    &\sum_{i=1}^{N} \wH\: A_{\ell}(\dist{\bx,\pH{i}})\\
    &\quad\le \ell^{d}\sum_{\pH{i}\in \ball{\bx,1/n}}\wH + \sum_{k=1}^{
    \floor{r n}-1}\sum_{\pH{i}\in\ball{\bx,k/n,(k+1)/n}}\wH\: \ell^{d}\left(\frac{\ell k}{n}\right)^{-(d+1)} + \ell^{-1} \sum_{\pH{i}\in \ball{\bx,\floor{r n}/n,r}}\wH\\
    &\quad\le c\: \ell^{d}\memf\bigl(\ball{\bx,1/n}\bigr) + c\:\ell^{-1} \sum_{k=1}^{\floor{rn}-1} \left(\frac{n}{k}\right)^{d+1} \memf\bigl(\ball{\bx,k/n,(k+1)/n}\bigr) + c\:\ell^{-1}\memf\bigl(\ball{\bx,\floor{rn}/n,r}\bigr)\\
    &\quad\le c_{d}\: \left(\frac{\ell}{n}\right)^{d},
\end{align*}
where the last inequality uses Assumption~\ref{assump:ball.vol} and $\memf\bigl(\ball{\bx,k/n,(k+1)/n}\bigr)=c_{d} \:k^{d-1}/n^d$.
\end{proof}

\begin{proof}[Lemma~\ref{lem:MZineq.mfd}] For $1\le p_{1}<\infty$, using \eqref{eq:fiapprox.reproduce.poly} and H\"{o}lder's inequality gives, for $P\in\polyspm$ and $\bx\in\mfd$,
\begin{equation}\label{eq:poly.conv.bound}
    |P(\bx)|^{p_{1}} \le \left(\int_{\mfd}|\vdh{n,\fiH}(\bx,\PT{z})||P(\PT{z})|^{p_{1}}\dmf{z}\right)\left(\int_{\mfd}|\vdh{n,\fiH}(\bx,\PT{z})|\dmf{z}\right)^{p_{1}-1}.
\end{equation}
Lemma~\ref{lem:fikerL1} shows that the second integral of the filtered kernel on the right-hand side is bounded. This with \eqref{eq:poly.conv.bound} gives
\begin{equation*}
    |P(\bx)|^{p_{1}} \le c\:\left(\int_{\mfd}|\vdh{n,\fiH}(\bx,\PT{z})||P(\PT{z})|^{p_{1}}\dmf{z}\right),
\end{equation*}
where the constant $c$ depends only on $d,p_{1},\fiH$ and $\fis$. Summing over quadrature nodes, we then obtain by Lemmas~\ref{lm:QH.loc.est} and \ref{lem:localization} that
\begin{align*}
    \sum_{i=1}^{N}\wH |P(\pH{i})|^{p_{1}}
    &\le c \int_{\mfd}|P(\PT{z})|^{p_{1}}\sum_{i=1}^{N}\wH |\vdh{n,\fiH}(\pH{i},\PT{z})|\dmf{z}\\
    &\le c \left(\max_{\PT{z}\in\mfd}\sum_{i=1}^{N}\wH \:A_{\ell}\bigl(\dist{\pH{i},\PT{z}}\bigr)\right) \norm{P}{\Lpm{p_{1}}}^{p_{1}}\\
    &\le c \left(\frac{\ell}{n}\right)^{d}\norm{P}{\Lpm{p_{1}}}^{p_{1}},
\end{align*}
where the constant $c$ depends only on $d,p_{1},\fiH$ and $\fis$.
\end{proof}

The proof of optimal-order error for filtered hyperinterpolation utilises its decomposition by framelets on manifolds \citep{wang2020tight,WaLeSlWo2017,WaSl2017}. 
Given $\fiH\in\CkRp$, $\fis\ge1$, we define recursively the \emph{contributions} of levels $j\in\Nz$ for $f\in \Lpm{p}$ by
\begin{equation}\label{eq:contrib}
    \contrib{0}(f) := \fiapprox{2^{-1}}(f) := 1,\quad \contrib{j}(f) := \fiapprox{2^{j-1}}(f) - \fiapprox{2^{j-2}}(f), \;\; j\in\N,
\end{equation}
The following lemma shows that $\contrib{j}(f)$ forms a decomposition of $f\in \Lpm{p}$, and it gives an upper bound of the $L_p$-norm of $\contrib{j}(f)$ for $f\in\sobm{p}{s}$.
\begin{lemma}\label{lm:contrib.decomp.nrm.BD} Let $1\le p\le \infty$, $d\ge2$, $s>0$. Then,
\begin{align}
    & \lim_{\ctrord\to\infty}\normB{\sum_{j=0}^{\ctrord}\contrib{j}(f)-f}{\Lpm{p}} = 0,\quad f\in\Lpm{p}, \label{eq:Lp.err.Uj}\\[1mm]
    & \normb{\contrib{j}(f)}{\Lpm{p}} \le c\: 2^{-js}\: \norm{f}{\sobm{p}{s}},\quad j\in \N, \; f\in\sobm{p}{s}, \label{eq:Lp.nrm.Uj}
\end{align}
where the constant $c$ depends only on $d$, $p$, $s$, $\fiH$ and $\fis$.
\end{lemma}
\begin{proof} For $f\in\Lpm{p}$, Theorem~\ref{thm:fiapprox.err.Lp} with \eqref{eq:contrib} gives
\begin{equation*}
    \normB{\sum_{j=0}^{\ctrord}\contrib{j}(f)-f}{\Lpm{p}} = \normb{\fiapprox{2^{\ctrord-1}}(f)-f}{\Lpm{p}}\le c_{d,\fiH,\fis}\: \bestapprox{2^{\ctrord-1}}{f}.
\end{equation*}
This with $\lim_{\ctrord\to\infty}\bestapprox{2^{\ctrord-1}}{f}=0$ gives \eqref{eq:Lp.err.Uj}.
For $f\in\sobm{p}{s}$ and $j\in\N$, Theorem~\ref{thm:fiapprox.err.Wp} with \eqref{eq:contrib} gives
\begin{align*}
    \normb{\contrib{j}(f)}{\Lpm{p}}
    &\le \normb{\fiapprox{2^{j-1}}(f)-f}{\Lpm{p}} + \normb{\fiapprox{2^{j-2}}(f)-f}{\Lpm{p}}\\
    &\le c\: 2^{-js}\:\norm{f}{\sobm{p}{s}},
\end{align*}
where the constant $c$ depends only on $d$, $p$, $s$, $\fiH$ and $\fis$.
\end{proof}

\begin{proof}[Theorem~\ref{thm:nondfh_clean_det}]  For $p=\infty$, Lemma~\ref{lm:fihyper.reproduce.p} with the linearity of $\fihyper{D^*,n}$ shows that for $q\in \polyspm$,
\begin{align}\label{eq:fiapprox.inf.nrm}
    \normb{f-\fihyper{D^*,n}(f)}{\Lpm{\infty}} &= \normb{(f-q)-\fihyper{D^*,n}(f-q)}{\Lpm{\infty}}\notag\\
    &\le \left(1+\normb{\fihyper{D^*,n}}{\ptop[\infty]}\right)\:\normb{f-q}{\Lpm{\infty}},
\end{align}
where using standard arguments,
\begin{equation}\label{eq:inf.inf.nrm.fihyper}
    \normb{\fihyper{D^*,n}}{\ptop[\infty]} := \sup_{\bx\in\mfd} \sum_{i=1}^{N} \wH \: |\fiker{n}(\bx,\pH{i})|.
\end{equation}
Taking the minimum over $q\in\polyspm$ of the right-hand side of \eqref{eq:fiapprox.inf.nrm} with Lemma~\ref{lem:best.approx.sob.mfd} gives
\begin{align*}
    \normb{f-\fihyper{D^*,n}(f)}{\Lpm{\infty}} &= \normb{(f-q)-\fihyper{D^*,n}(f-q)}{\Lpm{\infty}}\\
    &\le \left(1+\normb{\fihyper{D^*,n}}{\ptop[\infty]}\right)\:\bestapprox[\infty]{L}{f}\\
    &\le c_{d,s}\:\left(1+\normb{\fihyper{D^*,n}}{\ptop[\infty]}\right)\: n^{-s}\:\norm{f}{\sobm{\infty}{s}}.
\end{align*}
Since the quadrature rule $\QH$ is exact for degree $3n-1$, the condition of Lemma~\ref{lem:MZineq.mfd} is satisfied for $p_{0}=2$, then \eqref{eq:MZineq.p1} holds for $p_{1}=1$. This with \eqref{eq:inf.inf.nrm.fihyper} and Lemma~\ref{lem:fikerL1} gives
\begin{equation*}
    \normb{\fihyper{D^*,n}}{\ptop[\infty]} \le c_{d} \sup_{\bx\in\mfd}\int_{\mfd} |\fiker{n}(\bx,\by)|\dmf{y}\le c_{d,\fiH,\fis}.
\end{equation*}
Thus,
\begin{align*}
    \normb{f-\fihyper{D^*,n}(f)}{\Lpm{\infty}}
    \le c_{d,\fiH,\fis,s}\: n^{-s}\:\norm{f}{\sobm{\infty}{s}}.
\end{align*}

We next consider for $p\in [1,\infty)$.  Given $n\ge0$, let $m$ be the integer satisfying $2^{m}\le L<2^{m+1}$. Since $\contrib{j}(f)\in\polyspm[2^{j+1}]$, $\fihyper{D^*,n}$ reproduces $\contrib{j}(f)$ for $j\le m-1$, that is, $\fihyper{D^*,n}(\contrib{j}(f))=\contrib{j}(f)$, $j\le m-1$. Lemma~\ref{lm:contrib.decomp.nrm.BD} then gives
\begin{align}\label{eq:fihyper.err.via.contrib}
    \normb{f-\fihyper{D^*,n}(f)}{\Lpm{p}} &= \lim_{\ctrord\to\infty} \normB{\sum_{j=0}^{\ctrord}\contrib{j}\left(f-\fihyper{D^*,n}(f)\right)}{\Lpm{p}}\notag\\
    &= \lim_{\ctrord\to\infty} \normB{\sum_{j=m}^{\ctrord}\left(\contrib{j}(f)-\fihyper{D^*,n}(\contrib{j}(f))\right)}{\Lpm{p}}\notag\\
    &\le \sum_{j=m}^{\infty}\left(\norm{\contrib{j}(f)}{\Lpm{p}}+ \normb{\fihyper{D^*,n}(\contrib{j}(f))}{\Lpm{p}}\right).
\end{align}
To bound the right-hand side of the last inequality in \eqref{eq:fihyper.err.via.contrib}, we need the following estimate.
\begin{equation}\label{eq:fihyper.contrib.f.Lp.nrm}
    \normb{\fihyper{D^*,n}(\contrib{j}(f))}{\Lpm{p}} \le c\:\left(\frac{2^{j+1}}{n}\right)^{d/p} \normb{\contrib{j}(f)}{\Lpm{p}},
\end{equation}
where the constant $c$ depends only on $d$, $p$, $\fiH$ and $\fis$. For $p=1$ and $j\ge m$,
\begin{align*}
    \normb{\fihyper{D^*,n}(\contrib{j}(f))}{\Lpm{1}} &= \normB{\sum_{i=1}^{N}\wH\: \fiker{n}(\pH{i},\cdot)\:\contrib{j}(f;\pH{i})}{\Lpm{1}}\\
    &\le\sum_{i=1}^{N}\wH\: |\contrib{j}(f;\pH{i})|\:\norm{\fiker{n}(\pH{i},\cdot)}{\Lpm{1}}\\
    &\le c_{d,\fiH,\fis}\: \sum_{i=1}^{N}\wH\: |\contrib{j}(f;\pH{i})|\\
    &\le c_{d,\fiH,\fis}\: \left(\frac{2^{j+1}}{n}\right)^{d}\normb{\contrib{j}(f)}{\Lpm{1}},
\end{align*}
where the penultimate inequality uses Lemma~\ref{lem:fikerL1} and the last uses Lemma~\ref{lem:MZineq.mfd} with $p_{1}=1$. For $1<p<\infty$ and $j\ge m$, by H\"{o}lder's inequality,
\begin{align*}
    &\normb{\fihyper{D^*,n}(\contrib{j}(f))}{\Lpm{p}}^{p}\\
    &\quad = \int_{\mfd}\left|\sum_{i=1}^{N}\wH\: \fiker{n}(\bx,\pH{i})\:\contrib{j}(f;\pH{i})\right|^{p}\dmf{x}\\
    &\quad \le \int_{\mfd}\left(\sum_{i=1}^{N}\wH\: |\fiker{n}(\bx,\pH{i})|\:|\contrib{j}(f;\pH{i})|\right)^{p}\dmf{x}\\
    &\quad \le \int_{\mfd}\left(\sum_{i=1}^{N}\left(\wH\:|\fiker{n}(\bx,\pH{i})|\right)^{\frac{p-1}{p}}
    \left(\wH\: |\fiker{n}(\bx,\pH{i})|\:|\contrib{j}(f;\pH{i})|^{p}\right)^{\frac{1}{p}}\right)^{p}\dmf{x}\\
    &\quad \le \int_{\mfd}\left(\sum_{i=1}^{N}\wH\:|\fiker{n}(\bx,\pH{i})|\right)^{p-1}
    \left(\sum_{i=1}^{N}\wH\: |\fiker{n}(\bx,\pH{i})|\:|\contrib{j}(f;\pH{i})|^{p}\right)\dmf{x}.
\end{align*}
Using Lemma~\ref{lem:MZineq.mfd} with $p_{1}=1$ and $p_{1}=p$ and Lemma~\ref{lem:fikerL1} then gives
\begin{align*}
    &\normb{\fihyper{D^*,n}(\contrib{j}(f))}{\Lpm{p}}^{p}\\
    &\quad \le\left(c_{d}\:\max_{\bx\in\mfd}\normb{\fiker{n}(\bx,\cdot)}{\Lpm{1}}\right)^{p-1}
    \sum_{i=1}^{N}\wH\: |\contrib{j}(f;\pH{i})|^{p}\:\int_{\mfd}|\fiker{n}(\bx,\pH{i})|\dmf{x}\\
    &\quad \le c_{d,p,\fiH,\fis}\: \sum_{i=1}^{N}\wH\: |\contrib{j}(f;\pH{i})|^{p}\\
    &\quad \le c_{d,p,\fiH,\fis}\: \left(\frac{2^{j+1}}{n}\right)^{d} \normb{\contrib{j}(f)}{\Lpm{p}}^{p},
\end{align*}
which proves \eqref{eq:fihyper.contrib.f.Lp.nrm} for $p\in(1,\infty)$.
It follows from \eqref{eq:fihyper.err.via.contrib} and \eqref{eq:fihyper.contrib.f.Lp.nrm} that for $f\in \sobm{p}{s}$, $1\le p<\infty$, $s>d/p$,
\begin{align*}
    \normb{f-\fihyper{D^*,n}(f)}{\Lpm{p}}
    &\le c_{d,p,\fiH,\fis}\sum_{j=m}^{\infty}\left(1+\left(\frac{2^{j+1}}{n}\right)^{d/p}\right) \norm{\contrib{j}(f)}{\Lpm{p}}\\
    &\le c_{d,p,\fiH,\fis}\sum_{j=m}^{\infty}\left(1+\left(\frac{2^{j+1}}{n}\right)^{d/p}\right)2^{-js} \norm{f}{\sobm{p}{s}},
\end{align*}
where the second inequality uses \eqref{eq:Lp.nrm.Uj}, and where since $n\asymp 2^{m}$ and $s>d/p$,
\begin{align*}
    \sum_{j=m}^{\infty}\left(1+\left(\frac{2^{j+1}}{n}\right)^{d/p}\right)2^{-js}
    &\le c_{d,p,s}\:\sum_{j=m}^{\infty}\left(1+\left(\frac{2^{j+1}}{2^{m}}\right)^{d/p}\right)2^{-js}\\
    &\le c_{d,p,s}\:2^{-md/p}\:\sum_{j=m}^{\infty}\left(2^{md/p}+ 2^{(j+1)d/p}\right)2^{-js}\\
    &\le c_{d,p,s}\:2^{-md/p}\:\sum_{j=m}^{\infty}2^{-j(s-d/p)}\\
    &\le c_{d,p,s}\:2^{-ms}\\
    &\le c_{d,p,s}\:n^{-s},
\end{align*}
thus proving \eqref{eq:fihyper.Lp.err.Wp}.
\end{proof}

\begin{proof}[Theorem~\ref{thm:ndfh_ran_clean}]
Let $\{w_{i}\}_{i=1}^{|D^*|}$ be the real numbers computed in \eqref{eq:wi_dfh_ran}. Since $\{\bx_i\}_{i=1}^{|D^*|}$ is a set of random points on $\mfd$, we define four events, as follows. Let $
     \Omega_{D^*}$ be the event such that
     $\sum_{i=1}^{|D^*|}|w_{i}|^2\leq\frac{2}{|D^*|}$ and $\Omega_{D^*}^c$ be
the complement of $\Omega_{D^*}$, i.e. $\Omega_{D^*}^c$ be the event
$\sum_{i=1}^{|D^*|}|w_{i}|^2>\frac{2}{|D^*|}$. Let $\Xi_{D^*}$
the event that $\{(w_{i},\bx_i)\}_{i=1}^{|D^*|}$ is a
quadrature rule exact for polynomials in $\Pi_n^d$ associated with the measure $\nu$ and $\Xi_{D^*}^c$ the complement event of $\Xi_{D^*}$. Then, by Lemma \ref{lem:random_quadr},
\begin{equation}\label{eq:prob estimate for events}
     \mathbf P\{\Omega_{D^*}^c\}\leq\mathbf P\{\Xi_{D^*}^c\}\leq
     4\exp\left\{-C|D^*|/n^d\right\}.
\end{equation}
We write
\begin{align}
      &\mathbf{E}\left\{\|V_{D^*,n}-f^*\|^2_{L_2(\mfd)}\right\}\label{eq:total error}\\
      &\quad=\mathbf{E}\left\{\|V_{D^*,n}-f^*\|^2_{L_2(\mfd)}|\Omega_{D^*}\right\}\mathbf P\{\Omega_{D^*}\}
      +\mathbf{E}\left\{\|V_{D^*,n}-f^*\|^2_{L_2(\mfd)}|\Omega_{D^*}^c\right\}\mathbf P\{\Omega_{D^*}^c\}.\notag 
\end{align}
Under the event $\Omega_{D^*}^c$, using the weights in \eqref{eq:wi_dfh_ran}, we obtain that $V_{D^*,n}=0$. Then, by \eqref{eq:prob estimate for events},
\begin{equation}\label{eq:bd_no_event}
     \mathbf{E}\left\{\|V_{D^*,n}-f^*\|^2_{L_2(\mfd)}|\Omega_{D^*}^c\right\}\mathbf P\{\Omega_{D^*}^c\}
      \leq 4\|f^*\|^2_{L_\infty(\mfd)}\exp\bigl\{-C|D^*|/n^d\bigr\}.
\end{equation}
This together with \eqref{eq:prob estimate for events} gives
\begin{align*}
      &\mathbf{E}\left\{\|V_{D^*,n}-f^*\|^2_{L_2(\mfd)}\big|\Omega_{D^*}\right\}\notag\\
      &\quad=
      \mathbf{E}\left\{\int_{\mfd} \mathbf{E}\left\{(f^*(\bx)-V_{D^*,n}(\bx))^2\big|\Lambda_{D^*}\right\}\mathrm{d}\omega(\bx)
     \big|\Xi_{D^*},\Omega_{D^*}\right\}\mathbf P\{\Xi_{D^*}\}\\
     &\qquad+
     \mathbf{E}\left\{\int_{\mfd} \mathbf{E}\left\{(f^*(\bx)-V_{D^*,n}(\bx))^2\big|\Lambda_{D^*}\right\}\mathrm{d}\omega(\bx)
     \big|\Xi_{D^*}^c,\Omega_{D^*}\right\}\mathbf P\{\Xi_{D^*}^c\} \nonumber\\
     &=: \mathcal A_{D^*,n,1} + \mathcal A_{D^*,n,2}.
\end{align*}
To bound $\mathcal A_{D^*,n,1}$, we observe that when the event $\Omega_{D^*}\cap\Xi_{D^*}$ takes place, $\{w_{i}\}_{i=1}^{|D^*|}$ is a set of positive weights for quadrature rule $\mathcal Q_{|D^*|,n}$. We then obtain from Theorem~\ref{thm:nondfh_clean_det} and $f^*\in \mathbb W_2^r(\mfd)$ with $r>d/2$ that
\begin{equation*}
      \mathcal A_{D^*,n,1}\leq c_5^2 n^{-2r}\|f\|^2_{\mathbb W_2^r(\mfd)}.
\end{equation*}
On the other hand, under the event
$\Omega_{D^*}\cap\Xi_{D^*}^c$, by Cauchy-Schwarz inequality,
\begin{align*}
     \bigl(f^*(\bx)-V_{D^*,n}(\bx)\bigr)^2
      &\leq 2\|f^*\|_{L_\infty(\mfd)}^2+2\left|\sum_{i=1}^{|D^*|}w_{i}f^*(\bx_i)K_n(\bx_i,
       \bx)\right|^2\\
        &\leq
        2\|f^*\|_{L_\infty(\mfd)}^2+2\|f^*\|_{L_\infty(\mfd)}^2
        \sum_{i=1}^{|D^*|}w^2_{i}\sum_{i=1}^{|D^*|}|K_n(\bx_i, \bx)|^2.
\end{align*}
This with Lemma~\ref{lem:fikerL1} and \eqref{eq:prob estimate for events} gives
\begin{equation*}
      \mathcal A_{D^*,n,2}
     \leq
     2\|f^*\|_{L_\infty(\mfd)}^2(\mu(\mfd)+ 2c_1^2 n^{d})\exp\left\{-C|D^*|/n^d\right\}.
\end{equation*}
Then, with \eqref{eq:bd_no_event}, \eqref{eq:total error} and $cn^{d+\tau}\leq |D^*|\leq c' n^{2d}$, $\tau\in (0,d]$,
\begin{align}\label{eq:error VDnf}
      &\mathbf{E}\left\{\|V_{D^*,n}-f^*\|^2_{L_2(\mfd)}\right\}\notag\\
     &\quad\leq
     c_5^2 n^{-2r}\|f\|^2_{\mathbb W_2^r(\mfd)}+2\|f^*\|_{L_\infty(\mfd)}^2(2+\mu(\mfd)+ 2c_1^2 n^{d})\exp\left\{-C|D^*|/n^d\right\}\notag\\
     &\quad\leq C|D^*|^{-r/d},
\end{align}
thus completing the proof.
\end{proof}

\subsection{Proofs for Section~\ref{sec:ndfh_noisy}}\label{appendix:ndfh_noisy}
\begin{proof}[Theorem \ref{thm:ndfh_noisy_det}]
As $\mathbf{E}\{\epsilon_i\}=0$ for any $i=1,\dots,|D|$,
\begin{align*}
           \mathbf{E}\left\{V_{D,n}(\bx)\right\}
           &=
           \mathbf{E}\left\{\sum_{i=1}^mw_{i}y_iK_n(\bx_i,\bx)\right\}
            =
           \mathbf{E}\left\{\sum_{i=1}^mw_{i}(f^*(\bx_i)+\epsilon_i)K_n(\bx_i,\bx)\right\}\\
           &=
           \sum_{i=1}^mw_{i}f^*(\bx_i)K_n(\bx_i,\bx)
           +\sum_{i=1}^mw_{i}\mathbf{E}\{\epsilon_i\}K_n(\bx_i,\bx)
           =
           V^*_{D,n}(x),
\end{align*}
then,
\begin{equation}\label{eq:semipopulation 1}
           \mathbf{E}\left\{V^*_{D,n}(\bx)-V_{D,n}(\bx) \right\}=0.
\end{equation}
This implies
\begin{align}\label{eq:err split 1}
   &\mathbf{E}\left\{\|V_{D,n}-f^*\|^2_{L_2(\mfd)} \right\}\notag\\
     &\quad=
     \int_{\mfd} \mathbf{E}\{(f^*(x)-V_{D,n}(x))^2
     \}\dmf{x}\nonumber\\
    &\quad =
     \int_{\mfd} \mathbf{E}\{(f^*(x)-V^*_{D,n}(x)
     +V^*_{D,n}(x)-V_{D,n}(x))^2
     \}\dmf{x} \nonumber\\
    &\quad= \int_{\mfd}  (V^*_{D,n}(x)-f^*(x))^2 \dmf{x}
    + \int_{\mfd} \mathbf
    E\{(V^*_{D,n}(x)-V_{D,n}(x))^2 \}\dmf{x}
         \nonumber\\
         & \quad:=
         \mathcal{A}^\diamond_{D,n} + \mathcal{S}^\diamond_{D,n}.
\end{align}
For $\mathcal{A}^\diamond_{D,n}$ in \eqref{eq:err split 1}, Theorem~\ref{thm:nondfh_clean_det} gives
\begin{equation}\label{eq:bound A.1}
    \mathcal A^\diamond_{D,n}\leq
    c_5^2 \:n^{-2r}\|f^*\|^2_{\mathbb W_2^r(\mfd)}.
\end{equation}
To bound $\mathcal S^\diamond_{D,n}$, we observe from \eqref{eq:noisydats} that
\begin{align*}
         \mathbf
         E\left\{(V_{D^*,n}(\bx)-V_{D,n}(\bx))^2 \right\}
           &=
          \mathbf{E}\left\{\left(\sum_{i=1}^{|D|}(y_i-f^*(\bx_i))w_{i}
          K_n(\bx_i,\bx)\right)^2 \right\}\\
          &=
          \mathbf{E}\left\{\left(\sum_{i=1}^{|D|}\epsilon_iw_{i}
          K_n(\bx_i,\bx)\right)^2 \right\}\\
          &\leq M^2 \sum_{i=1}^{|D|}w_{i}^2|K_n(\bx_i,\bx)|^2,
\end{align*}
where the last inequality uses the independence of $\epsilon_1,\dots,\epsilon_{|D|}$.
This together with Lemma \ref{lem:fikerL1} and Assumption~\ref{assum:qrmfd} shows
\begin{align}\label{eq:bound S.1}
        \mathcal S^\diamond_{D,n}
         &\leq
         M^2 \int_{\mfd} \sum_{i=1}^{|D|} w_{i}^2|
         K_n(\bx_i,\bx)|^2 \dmf{x} \nonumber\\
          &=
          M^2 \sum_{i=1}^{|D|} w_{i}^2\int_{\mfd}|K_n(\bx_i,\bx)|^2 \dmf{x}
           \leq
          c_1M^2n^d \sum_{i=1}^{|D|} w_{i}^2 \leq
          \frac{c_1c_2^2M^2n^d}{|D|}.
\end{align}
Putting \eqref{eq:bound S.1} and \eqref{eq:bound A.1} to
\eqref{eq:err split 1}, we obtain
\begin{equation}\label{eq:IMportant.1}
     \mathbf{E}\left\{\|V_{D,n}-f^*\|^2_{L_2(\mfd)}\right\}
     \leq
     c_5^2n^{-2r}\|f^*\|^2_{\mathbb W_2^r(\mfd)}+ \frac{c_1c_2^2M^2n^d}{|D|},
\end{equation}
with $\frac{c_3}2 |D|^{\frac1{2r+d}}\leq n\leq c_3 |D|^{\frac1{2r+d}}$, then,
\begin{equation*}
      \mathbf{E}\left\{\|V_{D,n}-f^*\|^2_{L_2(\mfd)}\right\}\leq
      C_1|D|^{-\frac{2r}{2r+d}},
\end{equation*}
where $C_1:=4^r c_5^2 c_3^{-2r} \|f^*\|^2_{\mathbb W_2^r(\mfd)} + c_1 c_2^2 c_3^d M^2$, thus completing the proof. 
\end{proof}

We need the following Nikolski\^{\i}-type inequality for manifolds, as proved by \citet[Proposition~4.1]{FiMh2011}.
\begin{lemma}\label{lem:Nikolskii_ineq} For $n\in \Nz$ and $0<p<q\le\infty$, 
\begin{equation*}
	\|P_n\|_{L_q(\mfd)} \leq c\: n^{\frac{d}{p}-\frac{d}{q}}\|P_n\|_{L_p(\mfd)},
\end{equation*}
where the constant $c$ depends only on $d,p,q$.
\end{lemma}

We need the following concentration inequality, Lemma~\ref{Lemma:Concentration inequality 1.1}, established by \cite{WuZh2005}. 
Let $\mathcal{F}$ be a subset of a metric space. For $\varepsilon>0$, the \emph{covering number} $\mathcal{N}(\mathcal{F},\varepsilon)$ for $\mathcal{F}$ is the minimal natural integer $\ell$ such that $\mathcal{F}$ can be covered by $\ell$ balls of radius $\varepsilon$, see \cite{CuSm2002,Zhou2002}.

\begin{lemma}\label{Lemma:Concentration inequality 1.1}
Let $\mathcal{G}$ be a set of functions on a product space $X\times Y$ with Borel probability measure $\rho$.
For every
$g\in\mathcal{G}$, if $|g-\mathbf{E}g|\leq B$ almost everywhere and
$\mathbf{E}(g^2)\leq \tilde{c}(\mathbf{E} g)^\alpha$ for some $B\geq
0$, $0\leq\alpha\leq 1$ and $\tilde{c}\geq0$. 
Then, for any
$\varepsilon>0$,
\begin{equation*}
     \mathbf P\left\{\sup_{g\in\mathcal{G}}\frac{\left|\mathbf
     Eg-\frac1m\sum_{i=1}^mg(z_i)\right|}{\sqrt{(\mathbf
     Eg)^\alpha+\varepsilon^\alpha }}>\varepsilon^{1-\frac{\alpha}2}\right\}\leq2\mathcal N(\mathcal
     G,\varepsilon)\exp\left\{-\frac{m\varepsilon^{2-\alpha}}{2(\tilde{c}+\frac13B\varepsilon^{1-\alpha})}\right\},
\end{equation*}
where the expectation $\mathbf{E}g$ is taken on the product space $X\times Y$ with respect to $\rho$.
\end{lemma}

The third one is a covering number estimate for Banach space, as given by \cite{ZhJe2006}.
\begin{lemma}\label{Lemma:Covering number}
Let $\mathbb B$ be a finite-dimensional Banach space.
 Let $B_R$ be the closed ball of radius $R$
centered at origin given by
 $B_R:=\{f\in\mathbb B:\|f\|_{\mathbb B}\leq R\}$. Then,
\begin{equation*}
    \log\mathcal N(  B_R,\varepsilon)\leq
            \dim(\mathbb B) \log\left(\frac{4R}\varepsilon\right).
\end{equation*}
\end{lemma}

Let
$\mathcal X$ be a finite dimensional vector space endowed with norm
$\|\cdot\|_{\mathcal X}$, and $\mathcal Z\subset \mathcal X^*$ be a finite set. We say that $\mathcal Z$ is a \emph{norm generating set} for $\mathcal X$ if the mapping $T_{\mathcal Z}: \mathcal X\rightarrow\mathbb R^{|\mathcal Z|}$ defined by $T_{\mathcal Z}(x)=(z(x))_{z\in \mathcal Z}$ is injective. We call $T_{\mathcal Z}$ \emph{sampling operator}. Let $W:=T_{\mathcal Z}(\mathcal X)$ be the range of $T_{\mathcal Z}$, then the injectivity of $T_{\mathcal Z}$ implies that $T_{\mathcal Z}^{-1}:W\rightarrow \mathcal X$ exists.
Denote by $\|\cdot\|_{\mathbb R^{|\mathcal Z|}}$ the norm of $\mathbb R^{|\mathcal Z|}$, and $\|\cdot\|_{\mathbb R^{|\mathcal Z|^*}}$ the dual norm on $\mathbf R^{|\mathcal Z|^*}$ for $\|\cdot\|_{\mathbb R^{|\mathcal Z|}}$. We equip the space $W$ with the induced norm, and let $\|T_{\mathcal Z}^{-1}\|:=\|T_{\mathcal Z}^{-1}\|_{W\rightarrow \mathcal X}$ be the operator norm. In addition, let $\mathcal{K}_+$ be the positive cone of $\mathbb R^{|\mathcal Z|}$  {which is the set of all} $(r_z)_{z\in \mathcal{Z}}\in\mathbb R^{|\mathcal Z|}$ such that $r_z\geq0$. Then the following lemma \citep{MhNaWa2001} holds.

\begin{lemma}\label{Lemma:norming set}
Let $\mathcal Z$ be a norm generating set for $\mathcal X$, with
$T_{\mathcal Z}$ the corresponding sampling operator. If $\mathcal{L}\in
\mathcal X^*$ with $\|\mathcal{L}\|_{\mathcal X^*}\leq \mathcal A$, then there
exist positive numbers $\{a_z\}_{z\in \mathcal Z}$, depending only
on $\mathcal{L}$ such that for every $x\in\mathcal X$,
\begin{equation*}
    \mathcal{L}(x)=\sum_{z\in \mathcal Z}a_zz(x),
\end{equation*}
and
\begin{equation*}
    \|(a_z)\|_{\mathbb R^{|\mathcal Z|^*}}\leq \mathcal A\|T_{\mathcal Z}^{-1}\|.
\end{equation*}
If the space $W=T_{\mathcal Z}(X)$ contains an interior point $v_0\in \mathcal{K}_+$, and if $\mathcal{L}(T_{\mathcal Z}^{-1}v)\geq0$ when $v\in W\cap \mathcal{K}_+$, then we can choose $a_z\geq 0$.
\end{lemma}

\begin{proof}[Lemma \ref{lem:random_quadr}] 
For $p=1,2$, without loss of generality, we prove Lemma~\ref{lem:random_quadr} for $P_n\in\Pi_n^d$ satisfying $\|P_n\|_{L_{p,\nu}(\mfd)}= A$ for
some constant $A>0$. For arbitrary $P_n\in\Pi_{n}^{d}$ with
$\|P_n\|_{L_{p,\nu}(\mfd)}= A$, it follows from \eqref{eq:condition on distribution} and Lemma~\ref{lem:Nikolskii_ineq} that
\begin{equation*}
	\|P_n\|_{L_\infty(\mfd)}\leq
               \tilde{C}_1n^{\frac{d}{p}}\|P_n\|_{L_p(\mfd)}
               \leq c^{1/p}_4\tilde{C}_1n^{\frac{d}{p}}\|P_n\|_{L_{p,\nu}(\mfd)},
\end{equation*}
 and
\begin{align*}
  \mathbf{E}\left\{|P_n|^{2p}\right\}
  =\int_{\mfd}|P_n(\bx)|^{2p}{\rm d}\nu(\bx)
  &\leq \|P_n\|^p_{L_\infty(\mfd)}\int_{\mfd}|P_n(\bx)|^p{\rm d}\nu(\bx)\\
  &\leq c_4(\tilde{C}_1)^pn^d\|P_n\|^p_{L_{p,\nu}(\mfd)} \mathbf{E}\left[|P_n|^p\right].
\end{align*}
Let $g(z_i)=|P_n(\bx_i)|^p$,
$B=2c_4(\tilde{C}_1)^pn^d\|P_n\|^p_{L_{p,\nu}(\mfd)}$,
$\tilde{c}=c_4(\tilde{C}_1)^pn^d\|P_n\|^p_{L_{p,\nu}(\mfd)}$, $m=N$, $\alpha=1$ and $
\mathcal{G}_p=\{|P_n|^p:P_n\in\Pi_n^d,\|P_n\|_{L_{p,\nu}(\mfd)}=A\}$ in Lemma
\ref{Lemma:Concentration inequality 1.1}. Then, for any
$\varepsilon>0$,
\begin{align*}
     &\mathbf P\left\{\sup_{P_n\in\Pi_n^{d},\|P_n\|_{L_{p,\nu}(\mfd)}=A}
     \frac{\left|\|P_n\|_{L_{p,\nu}(\mfd)}^p-\frac1{N}\sum_{i=1}^{N}
     |P_n(\bx_i)|^p\right|}{
     \sqrt{\|P_n\|_{L_{p,\nu}(\mfd)}^p+\varepsilon}}>\sqrt{\varepsilon}\right\}\\
     &\quad \leq 2\mathcal N(\mathcal{G}_p,\varepsilon)
     \exp\left\{
     -\frac{{N}\varepsilon}{\tilde{C}_2n^d A^p}\right\},
\end{align*}
where $\tilde{C}_2=10c_4(\tilde{C}_1)^p/3$.

We need to estimate the above covering number $\mathcal N(\mathcal{G}_p,\varepsilon)$ for $p=1,2$. To this end, we let
$\mathcal{G}_1':=\{P_n\in\Pi_n^d:\|P_n\|_{L_{p,\nu}(\mfd)}=A\}$ and $\mathcal{G}_2':=\{P_n\in\Pi_{2n}^d:\|P_n\|_{L_{p,\nu}(\mfd)}=A\}$.
By definition,
$\mathcal N(\mathcal{G}_1,\varepsilon)\leq \mathcal N(\mathcal{G}_1',\varepsilon)$ and $\mathcal N(\mathcal{G}_2,\varepsilon)=\mathcal N(\mathcal{G}_2',\varepsilon)$,
where for $p=2$, we have used $|P_n|^2\in\Pi_{2n}^d$.
Then, by Lemma \ref{Lemma:Covering
number}, for $p=1,2$, we obtain the upper bound
\begin{align*}
     &\mathbf P\left\{\sup_{P_n\in\Pi_n^{d},\|P_n\|_{L_{p,\nu}(\mfd)}=A}
     \frac{\left|\|P_n\|_{L_{p,\nu}(\mfd)}^p-\frac1{N}\sum_{i=1}^{N}
     |P_n(\bx_i)|^p\right|}{
     \sqrt{\|P_n\|_{L_{p,\nu}(\mfd)}^p+\varepsilon}}>\sqrt{\varepsilon}\right\}\\
     &\quad\leq 2\exp\left\{(2n)^d\log\frac{4A^p}{\varepsilon}
     -\frac{{N}\varepsilon}{\tilde{C}_2n^dA^p}\right\},
\end{align*}
where we have used the estimate $\dim \mathcal{G}_p\leq (pn)^d$.
 Let
$\varepsilon=A^p/4$. As $N/n^{2d}>c$ for sufficiently large constant $c$, with confidence $1-2\exp\left\{-C N/n^d\right\}$,
there holds
\begin{equation*}
    \left|\|P_n\|_{L_{p,\nu}(\mfd)}^p-\frac1{N}\sum_{i=1}^{N}|P_n(\bx_i)|^p\right|
     \leq \frac{\sqrt{5}}{4}\|P_n\|^p_{L_{p,\nu}(\mfd)}.
\end{equation*}
From this, we then obtain
\begin{equation}\label{MZ}
           \frac{1}{3}\|P_n\|_{L_{p,\nu}(\mfd)}^p\leq\frac1{N}\sum_{i=1}^{N}
           |P_n(\bx_i)|^p\leq\frac{5}{3}\|P\|_{L_{p,\nu}(\mfd)}^p\quad
           \forall P_n \in\Pi_n^{d},\ p=1,2
\end{equation}
holds with probability at least $1-2\exp\left\{-C N/n^d \right\}$.

We now apply \eqref{MZ} with $p=2$ and Lemma \ref{Lemma:norming set}
to prove Lemma \ref{lem:random_quadr}. In Lemma
\ref{Lemma:norming set}, we take $\mathcal X=\Pi_n^{d}$,
$\|P_n\|_{\mathcal X}=\|P_n\|_{L_{2,\nu}(\mfd)}$, and $\mathcal Z$ the
set of point evaluation functionals $\{\delta_{\mathbf
x_i}\}_{i=1}^{N}$. The operator $T_{\mathcal Z}$ is then the
restriction map $P_n\mapsto P_n|_{X_N} $  and
$\|f\|_{\Lambda_{D},2}:=\left(\frac1N\sum_{i=1}^N|f(\bx_i)|^2\right)^\frac{1}{2}$.
 It follows from \eqref{MZ} that with confidence at least
$1-2\exp\left\{-\tilde{C}_3 N/n^d \right\}$,
there holds $\|T_{\mathcal Z}^{-1}\|\leq \sqrt{\frac{5}{3}}$. 
We let $\mathcal{L}$ be the functional
\begin{equation*}
    \mathcal{L}: P_n\mapsto \int_{\mfd}P_n(x)\mathrm{d}\nu(x).
\end{equation*}
By H\"{o}lder inequality, $\|y\|_{\mathcal X^*}\leq 1$.
Lemma \ref{Lemma:norming set} then shows that there exists a set of real
numbers $\{w_{i,n}\}_{i=1}^N$ such that
\begin{equation*}
    \int_{\mfd}P_n(x)\mathrm{d}\nu(x)=\sum_{i=1}^{N}w_{i,n}P_n(\bx_i)
\end{equation*}
holds with confidence at least
$1-2\exp\left\{-\tilde{C}_3 N/n^d \right\}$, subject to
\begin{equation*}
    \frac1N\sum_{i=1}^{N}\left(\frac{w_{i,n} }{1/{N}}\right)^2\leq2.
\end{equation*}

Finally, we use the second assertion of Lemma \ref{Lemma:norming
set} and \eqref{MZ} with $p=1$ to prove  the positivity of $w_{i,n}$. Since $1\in \Pi_n^d$, we have $v_0:=1|_{X_N}$ is an interior point of $\mathcal
K_+$. For $P_n\in\Pi_n^d$, $T_{\mathcal
Z}P_n=P_n|_{X_N}$ is in $W\cap\mathcal{K}_+$ if and only if
$P_n(\bx_i)\geq 0$ for all $\bx_i\in X_N$. For arbitrary
$P_n$ satisfying $P_n(\bx_i)\geq0$, $\bx_i\in X_N$,
define $\xi_i(P_n)=P_n(\bx_i)$. From Lemma
\ref{lem:Nikolskii_ineq} and \eqref{eq:condition on distribution}, we obtain for $i=1,\dots,N$,
\begin{align*}
    &|\xi_i| \leq \|P_n\|_{L_\infty(\mfd)}
    \leq \tilde{C}_1n^d\|P_n\|_{L_1(\mfd)}
    \leq \tilde{C}_1c_4n^d\|P_n\|_{L_{1,\nu}(\mfd)},\\
    &|\xi_i-\mathbf{E}\xi_i| \leq 2\|P_n\|_{L_\infty(\mfd)}
    \leq 2\tilde{C}_1c_4n^d\|P_n\|_{L_{1,\nu}(\mfd)},\\
    &\mathbf{E}\xi_i^2 \leq \|P_n\|_{L_\infty(\mfd)}
    \|P_n\|_{L_{1,\nu}(\mfd)}\leq \tilde{C}_1 c_4 n^d \|P_n\|^2_{L_{1,\nu}(\mfd)}.
\end{align*}
Applying Lemma \ref{Lemma:Concentration inequality 1.1} with
$B=2 \tilde{C}_1 c_4 n^d\|P_n\|_{L_{1,\nu}(\mfd)}$,
$\tilde{c}=\tilde{C}_1 c_4 n^d\|P_n\|^2_{L_{1,\nu}(\mfd)}$ and $\alpha=0$ to the set
$\{P_n:P_n\in\Pi_n^d, \|P_n\|_{L_{1,\nu}(\mfd)}=A\}$, using Lemma
\ref{Lemma:Covering number}, we obtain for any $\varepsilon>0$,
\begin{equation*}
     \mathbf P\left\{\sup_{P_n\in\Pi_n^{d}, P_n=|P_n| \atop \|P_n\|_{L_{1,\nu}(\mfd)}=A}
     \left| y(P_n)-\frac1{N}
     \sum_{i=1}^{N} P_n(\bx_i) \right|>\varepsilon\right\}
   	\leq
     2\exp\left\{n^d\log\frac{4A }{\varepsilon}
     -\frac{{N}\varepsilon^2}{2\tilde{C}_1c_4n^d A (A +2\varepsilon/3)}\right\}.
\end{equation*}
Let $\varepsilon=A/4=\frac{1}{4}\|P_n\|_{L_{1,\nu}(\mfd)}$. We then obtain that with
confidence
$1-2\exp\left\{-C N/n^d\right\}$,
\begin{equation*}
    \left|y(P_n)-\frac1{N}
     \sum_{i=1}^{N} P_n(\bx_i)\right|\leq
     \frac14\|P_n\|_{L_{1,\nu}(\mfd)},
\end{equation*}
This and \eqref{MZ} imply that for $P_n$ which satisfies
that $P_n(\bx_i)\geq0$ for all $\bx_i\in X_N$, the inequality
$$
   \left|y(P_n)-\frac1{N}
     \sum_{i=1}^{N} P_n(\bx_i)\right|\leq \frac{3}{4}\frac1{N}\sum_{i=1}^{N} P_n(\bx_i)
$$
holds with confidence
$1-4\exp\left\{-C N/n^d\right\}$
with the constant $C$ depending only on $\tilde{C}_3$ and $c_4$, then,
\begin{equation*}
    y(P_n)\geq \frac{1}{4} \frac1{N}\sum_{i=1}^{N} P_n(\bx_i)\geq0
\end{equation*}
for arbitrary $P_n\in\Pi_n^d$ satisfying $P_n(\bx_i)\geq0$,
$\bx_i\in X_N$. Lemma \ref{Lemma:norming
set} then implies $w_{i,n}\geq 0$ with confidence $1-4\exp\left\{-C N/n^d\right\}$, thus completing the proof of Theorem~\ref{lem:random_quadr}.
\end{proof}

\begin{proof}[Theorem \ref{thm:ndfh_noisy_ran}]
Let $\{w_{i}\}_{i=1}^{|D|}$ be the real weights in \eqref{eq:wi_dfh_ran}. Since $\{\bx_i\}_{i=1}^{|D|}$ is a set of random points on $\mfd$, we define four events, as follows. Let $
     \Omega_{D}$ be the event such that
     $\sum_{i=1}^{|D|}|w_{i}|^2\leq\frac{2}{|D|}$ and $\Omega_{D}^c$ be
the complement of $\Omega_D$, i.e. $\Omega_{D}^c$ be the event
$\sum_{i=1}^{|D|}|w_{i}|^2>\frac{2}{|D|}$. Let $\Xi_D$
the event that $\{(w_{i},\bx_i)\}_{i=1}^{|D|}$ is a
quadrature rule exact for polynomials in $\Pi_n^d$ associated with the measure $\nu$ and $\Xi_D^c$ the complement event of $\Xi_D$. Then, by Lemma \ref{lem:random_quadr},
\begin{equation}\label{Probability for events}
     \mathbf P\{\Omega_{D}^c\}\leq\mathbf P\{\Xi_D^c\}\leq
     4\exp\left\{-C|D|/n^d\right\}.
\end{equation}
To estimate the approximation error, we write
\begin{align}
      &\mathbf{E}\left\{\|V_{D,n}-f^*\|^2_{L_2(\mfd)}\right\}\label{expectation decomposition}\\
      &\quad=\mathbf{E}\left\{\|V_{D,n}-f^*\|^2_{L_2(\mfd)}|\Omega_{D}\right\}\mathbf P\{\Omega_{D}\}
      +\mathbf{E}\left\{\|V_{D,n}-f^*\|^2_{L_2(\mfd)}|\Omega_{D}^c\right\}\mathbf P\{\Omega_{D}^c\}.\notag 
\end{align}
Under the event $\Omega_{D}^c$, we obtain from \eqref{eq:wi_dfh_ran} that $V_{D,n}=0$. Then, by \eqref{Probability for events},
\begin{equation}\label{bound without event}
     \mathbf{E}\left\{\|V_{D,n}-f^*\|^2_{L_2(\mfd)}|\Omega_D^c\right\}\mathbf P\{\Omega_{D}^c\}
      \leq 4\|f^*\|^2_{L_\infty(\mfd)}\exp\bigl\{-C|D|/n^d\bigr\}.
\end{equation}
Next, we estimate the first term of RHS of \eqref{expectation decomposition}
when the event $\Omega_{D}$ takes place.
Let $\Lambda_D:=\{\bx_i\}_{i=1}^{|D|}$ be the set of points of data $D$. By the independence between $\{\epsilon_i\}_{i=1}^{|D|}$ and $\Lambda_{D}$ and $\mathbf{E}\{\epsilon_i\}=0$ for $i=1,\dots, |D|$,
\begin{align*}
           \mathbf{E}\left\{V_{D,n}(\bx)\big|\Lambda_{D}\right\}
           &=
           \mathbf{E}\left\{\sum_{i=1}^mw_{i}y_i
           K_n(\bx_i, \bx)\big|\Lambda_{D}\right\}\\
           &=
           \mathbf{E}\left\{\sum_{i=1}^mw_{i}(f^*(\bx_i)+\epsilon_i)K_n(\bx_i, \bx)\big|\Lambda_{D}\right\}\\
           &=
           \sum_{i=1}^mw_{i}f^*(\bx_i)K_n(\bx_i, \bx)+\sum_{i=1}^mw_{i}\mathbf{E}\{\epsilon_i\}K_n(\bx_i, \bx)\\[1mm]
           &=
           V_{D^*,n}(\bx).
\end{align*}
Hence,
\begin{equation}\label{condition expectation}
           \mathbf{E}\left\{\left(V_{D^*,n}(\bx)-V_{D,n}(\bx)\right)\big|\Lambda_{D}\right\}=0.
\end{equation}
This allows us to write
\begin{align}\label{decomposition 1}
   &\mathbf{E}\left\{\|V_{D,n}-f^*\|^2_{L_2(\mfd)}\big|\Omega_{D}\right\}
     =
    \mathbf{E}\left\{\int_{\mfd}
     \mathbf{E}\{(f^*(\bx)-V_{D,n}(\bx))^2
    \big|\Lambda_{D}\}\mathrm{d}\omega(\bx)\big|\Omega_{D}\right\} \nonumber\\
    &\quad=
    \mathbf{E}\left\{\int_{\mfd} \mathbf{E}\{(f^*(\bx)-V_{D^*,n}(\bx)
    +V_{D^*,n}(\bx)-V_{D,n}(\bx))^2
    \big|\Lambda_{D}\}\mathrm{d}\omega(\bx)\big|\Omega_{D}\right\} \nonumber\\
    &\quad=
    \mathbf{E}\left\{ \int_{\mfd} \mathbf
    E\{(V_{D^*,n}(\bx)-V_{D,n}(\bx))^2\big|\Lambda_{D}\}\mathrm{d}\omega(\bx)
    \big|\Omega_{D}\right\} \nonumber \\
        &\qquad+
         \mathbf{E}\left\{ \int_{\mfd} \mathbf
         E\{(V_{D^*,n}(\bx)-f^*(\bx))^2
         \big|\Lambda_{D}\}\mathrm{d}\omega(\bx)\big|\Omega_{D}\right\}\nonumber\\
         &\quad :=
         \mathcal S_{D,n}+\mathcal A_{D,n}.
\end{align}
Given $\Lambda_{D}$, if the event  $\Omega_{D}$ occurs, by $|\epsilon_i|\leq M$,
\begin{align*}
          \mathbf
         E\left\{(V_{D^*,n}(\bx)-V_{D,n}(\bx))^2\big|\Lambda_{D}\right\}
         & =
          \mathbf{E}\left\{\left(\sum_{i=1}^{|D|}
          \epsilon_i w_{i}K_n(\bx_i, \bx)\right)^2
          \bigg|\Lambda_{D}\right\}\\
          &\leq
           M^2  \sum_{i=1}^{|D|}w_{i}^2 |K_n(\bx_i, \bx)|^2,
\end{align*}
where the second line uses the independence of $\epsilon_1,
\dots,\epsilon_{|D|}$.
This with Lemma~\ref{lem:fikerL1} shows
\begin{align}\label{bound S}
        \mathcal S_{D,n}
         &\leq
         M^2\mathbf{E}\left\{ \int_{\mfd}  \sum_{i=1}^{|D|}
         w_{i}^2|K_n(\bx_i, \bx)|^2
         \mathrm{d}\omega(\bx) \big|\Omega_D\right\}\nonumber\\
          &=
          M^2\mathbf{E}\left\{\sum_{i=1}^{|D|} w_{i}^2\int_{\mfd}
          |K_n(\bx_i, \bx)|^2 \mathrm{d}\omega(\bx)\big|\Omega_D\right\} \nonumber \\
          &\leq
          c_1^2 M^2n^d\mathbf{E}\left\{\sum_{i=1}^{|D|} w_{i}^2\right\}\leq
         \frac{2c_1^2M^2n^d}{|D|}.
\end{align}
On the other hand, similar as the derivation of \eqref{eq:error VDnf}, we obtain
\begin{equation}\label{bound a}
      \mathcal A_{D,n}\leq
      c_5^2n^{-2r}\|f^*\|^2_{\mathbb W_2^r(\mfd)}+ 2\|f^*\|_{L_\infty(\mfd)}^2(\mu(\mfd)+ 2c_1^2 n^{d})\exp\{-C|D|/n^d\}.
\end{equation}
This and \eqref{bound S} and \eqref{decomposition 1} give
\begin{align*}
      &\mathbf{E}\left\{\|V_{D,n}-f^*\|^2_{L_2(\mfd)}\big|\Omega_{D}\right\}\\
      &\;\leq
        c_5^2n^{-2r}\|f\|^2_{\mathbb W_2^r(\mfd)}+2\|f^*\|_{L_\infty(\mfd)}^2(\mu(\mfd)+ 2c_1^2 n^{d})\exp\{-C|D|/n^d\}+\frac{2c_1^2M^2n^d}{|D|}.\notag
\end{align*}
Putting the above estimate and \eqref{bound without event} into
\eqref{expectation decomposition}, we obtain
\begin{align}
    \mathbf{E}\left\{\|V_{D,n}-f^*\|^2_{L_2(\mfd)}\right\}
    &\leq
     c_5^2n^{-2r}\|f\|^2_{\mathbb W_2^r(\mfd)}+\frac{2c_1^2M^2n^d}{|D|}\notag\\
    &\qquad + 2\|f^*\|_{L_\infty(\mfd)}^2(\mu(\mfd)+ 2c_1^2 n^{d}+2)\exp\bigl\{-C|D|/n^d\bigr\}.\label{IMportant}
\end{align}
Taking account of $n\asymp |D|^{\frac1{2r+d}}$ and $r>d/2$, we then have
\begin{equation*}
    n^{d}\exp\left\{-C|D|/n^d\right\}\leq
       C'_5 |D|^{\frac{d}{2r+d}}\exp\left\{-C|D|^{\frac{2r}{2r+d}}\right\}
       \leq\tilde{C}_5|D|^{-\frac{2r}{2r+d}}.
\end{equation*}
Thus,
\begin{equation*}
    \mathbf{E}\left\{\|V_{D,n}-f^*\|^2_{L_2(\mfd)}\right\}\leq
      C_3|D|^{-\frac{2r}{2r+d}}
\end{equation*}
with $C_3$ a constant independent of $|D|$, thus completing
the proof.
\end{proof}

\subsection{Proofs for Section~\ref{sec:dfh_clean}}\label{appendix:dfh_clean}
\begin{proof}[Theorem~\ref{thm:dfh_clean_det}]
By Definition~\ref{defn:dfh} and Theorem~\ref{thm:nondfh_clean_det}, and $\sum_{j=1}^m\frac{|D_j|}{|D|}=1$, for $f^* \in \sobm{2}{r}$,
\begin{align*}
	\normb{V_{D^*,n}^{(m)}-f^*}{\Lpm{2}}
	&\leq \sum_{j=1}^m\frac{|D^*_j|}{|D^*|}\normb{V_{D^*_j,n}-f^*}{\Lpm{2}}\\
	&\leq c_5^2 n^{-r} \norm{f^*}{\sobm{2}{r}},
\end{align*}
thus completing the proof.
\end{proof}

\begin{proof}[Theorem~\ref{thm:dfh_ran_clean}]
For each $j=1,\dots,m$, by Theorem~\ref{thm:ndfh_ran_clean} with $\min_{j=1,\dots,m}|D^*_j|\geq c n^{d+\tau}$, $\tau \in (0,d]$, and also $|D^*_j|\leq |D^*|\leq c'n^{2d}$,
\begin{align*}
    \mathbf{E}\left\{\bigl\|V_{D^*_j,n}-f^*\bigr\|_{L_2(\mfd)}^2\right\}
        \leq Cn^{-r},
\end{align*}
Then, for a partition $\{D^*_j\}_{j=1}^m$ of $D^*$, by Jensen's inequality,
\begin{align*}
    \mathbf{E}\left\{\bigl\|V_{D^*,n}^{(m)}-f^*\bigr\|_{L_2(\mfd)}^2\right\}
        \leq \sum_{j=1}^m\frac{|D^*_j|}{|D|}
        \mathbf{E}\left\{\bigl\|V_{D^*_j,n}-f^*\bigr\|_{L_2(\mfd)}^2\right\}
        \leq Cn^{-2r} \leq C |D^*|^{-r/d}.
\end{align*}
\end{proof}

\subsection{Proofs for Section~\ref{sec:dfh_noisy}}\label{appendix:dfh_noisy}
To prove Theorem~\ref{thm:dfh_noisy_det}, we need the following modified version of
\citet[Proposition 4]{GuLiZh2017}.

\begin{lemma}\label{Lemma:distributed.1}
 For $V_{D,n}^{(m)}$ in Definition~\ref{defn:dfh} with quadrature rule given by \eqref{eq:qrmfd}, there holds
\begin{align}
    &\mathbf{E}\left\{\bigl\|V_{D,n}^{(m)}-f^*\bigr\|_{L_2(\mfd)}^2\right\}\notag\\
        &\quad\leq
         \sum_{j=1}^m\frac{|D_j|^2}{|D|^2}\mathbf{E}\left\{\|V_{D_j,n}
         -f^*\|_{L_2(\mfd)}^2\right\}
        +\sum_{j=1}^m\frac{|D_j|}{|D|}
        \bigl\| V_{D^*_j,n}-f^*\bigr\|_{L_2(\mfd)}^2,\label{eq:distributed.bound}
\end{align}
where $V_{D^*_j,n}$ is given by \eqref{eq:VDnf}.
\end{lemma}

\begin{proof} Due to \eqref{eq:dfh} and
$\sum_{j=1}^m\frac{|D_j|}{|D|}=1$, we have
\begin{align*}
       &\bigl\|V_{D,n}^{(m)}-f^*\bigr\|_{L_2(\mfd)}^2
        =
       \left\|\sum_{j=1}^m\frac{|D_j|}{|D|}(V_{D_j,n}-f^*)\right\|_{L_2(\mfd)}^2\\
       &\hspace{-2mm}=
       \sum_{j=1}^m\frac{|D_j|^2}{|D|^2}\|V_{D_j,n}-f^*\|_{L_2(\mfd)}^2
       +
       \sum_{j=1}^m\frac{|D_j|}{|D|}\left\langle
       V_{D_j,n}-f^*,\sum_{k\neq
       j}\frac{|D_k|}{|D|}(V_{D_k,n}-f^*)\right\rangle_{L_2(\mfd)}.
\end{align*}
Taking expectations gives
\begin{align*}
        &\mathbf{E}\left\{\bigl\|V_{D,n}^{(m)}-f^*\bigr\|_{L_2(\mfd)}^2\right\}
        =
        \sum_{j=1}^m\frac{|D_j|^2}{|D|^2}\mathbf{E}\left\{\|V_{D_j,n}
        -f^*\|_{L_2(\mfd)}^2\right\}\\
        &\quad+
        \sum_{j=1}^m\frac{|D_j|}{|D|}\left\langle
       \mathbf{E}_{D_j}\bigl\{V_{D_j,n}\bigr\}-f^*,\mathbf{E}\left\{V_{D,n}^{(m)}\right\}-f^*- \frac{|D_j|}{|D|}
       \left(\mathbf{E}_{D_j}\{V_{D_j,n}\}-f^*\right)\right\rangle_{L_2(\mfd)}.
\end{align*}
Here,
\begin{equation*}
       \sum_{j=1}^m\frac{|D_j|}{|D|}\left\langle
       \mathbf{E}_{D_j}\{V_{D_j,n}\}-f^*,\mathbf{E}\left\{V_{D,n}^{(m)}\right\}-f^*\right\rangle_{L_2(\mfd)}
        =\left\|\mathbf{E}\left\{V_{D,n}^{(m)}\right\}-f^*\right\|_{L_2(\mfd)}^2.
\end{equation*}
Then,
\begin{align*}
        \mathbf
        E\left\{\bigl\|V_{D,n}^{(m)}-f^*\bigr\|_{L_2(\mfd)}^2\right\}
          &=
         \sum_{j=1}^m\frac{|D_j|^2}{|D|^2}\mathbf{E}\left\{\|V_{D_j,n}-f^*\|_{L_2(\mfd)}^2\right\}\\
         &\qquad-\sum_{j=1}^m\frac{|D_j|^2}{|D|^2}\left\|\mathbf{E}\{V_{D_j,n}\}-f^*\right\|_{L_2(\mfd)}^2 +
         \left\|\mathbf{E}\left\{V_{D,n}^{(m)}\right\}-f^*\right\|_{L_2(\mfd)}^2.
\end{align*}
By \eqref{eq:semipopulation 1},
$$
      \mathbf
      E\left\{V_{D,n}^{(m)}\right\}=\sum_{j=1}^m\frac{|D_j|}{|D|}V_{D^*_j,n}.
$$
This plus
$\sum_{j=1}^m\frac{|D_j|}{|D|}=1$ gives
\begin{align*}
    \left\|\mathbf{E}\left\{V_{D,n}^{(m)}\right\}-f^*\right\|^2_{L_2(\mfd)}
    &=
       \left\|\sum_{j=1}^m\frac{|D_j|}{|D|}\left(V_{D^*_j,n}-f^*\right)\right\|_{L_2(\mfd)}^2\\[1mm]
    &\leq
       \sum_{j=1}^m\frac{|D_j|}{|D|}\bigl\|V_{D^*_j,n}-f^*\bigr\|_{L_2(\mfd)}^2,
\end{align*}
thus proving the bound in \eqref{eq:distributed.bound}.
\end{proof}

\begin{proof}[Theorem \ref{thm:dfh_noisy_det}]
By Lemma \ref{Lemma:distributed.1}, we
only need to estimate the bounds of $\mathbf
E\left\{\|V_{D_j,n}-f^*\|_{L_2(\mfd)}^2\right\}$ and $\bigl\|V_{D^*_j,n}-f^*\bigr\|_{L_2(\mfd)}^2$.
By $\min_{j=1,\dots,m}|D_j|\geq |D|^{2d/(2r+d)}$ and $D_j$ and Assumption~\ref{assum:qrmfd}, there exists a quadrature rule for each local server which is exact for polynomials of degree $3n-1$ for $n$ satisfying $\frac{c_3}6|D|^{1/(2r+d)}\leq n\leq \frac{c_3}3|D|^{1/(2r+d)}$.

By \eqref{eq:IMportant.1}, for $j=1,\dots,m$,
\begin{equation*}
    \mathbf{E}\left\{\|V_{D_j,n}-f^*\|^2_{L_2(\mfd)}\right\} \leq c_5^2n^{-2r}\|f^*\|^2_{\mathbb W_2^r(\mfd)}+ \frac{c_1c_2^2M^2n^d}{|D_j|}.
\end{equation*}
This gives
\begin{align}\label{bound first term.1}
     &\sum_{j=1}^m\frac{|D_j|^2}{|D|^2}\mathbf{E}\left\{\|V_{D_j,n}-f^*\|_{L_2(\mfd)}^2\right\}  \nonumber\\
     &\quad\leq 36^r c_5^2c_3^{-2r}\|f^*\|^2_{\mathbb W_2^r(\mfd)}|D|^{-\frac{2r}{2r+d}} + 3^{-d} c_1 c_2^2 c_3^{d}M^2  \sum_{j=1}^m\frac{|D_j|^2}{|D|^2}\frac{|D|^{\frac{d}{2r+d}}}{|D_j|}\notag\\
     &\quad = C_1|D|^{-\frac{2r}{2r+d}},
\end{align}
where $C_1:=36^r c_5^2c_3^{-2r}\|f^*\|^2_{\mathbb W_2^r(\mfd)} + 3^{-d}c_1 c_2^2 c_3^d M^2$. 
 
For each $j=1,\dots,m$, Assumption~\ref{assum:qrmfd} implies that there exists a quadrature rule with nodes of $D_j$ and $|D_j|$ positive weights such that $V_{D^*_j,n}$ is a filtered hyperinterpolation for the noise-free data set $\{\bx_i,f^*(\bx_i)\}_{\bx_i\in D_j}$. Theorem~\ref{thm:nondfh_clean_det} then gives
\begin{equation*}
    \left\| V_{D^*_j,n} -f^*\right\|_{L_2(\mfd)}^2
         \leq c_5^2n^{-2r}\|f^*\|^2_{\mathbb W_2^r(\mfd)} \quad \forall j=1,2,\dots,m.
\end{equation*}
This together with conditions
$\sum_{j=1}^m\frac{|D_j|}{|D|}=1$ and $n\geq \frac{c_3}6|D|^{1/(2r+d)}$
gives
\begin{equation}\label{bound second term.1}
     \sum_{j=1}^m\frac{|D_j|}{|D|}\left\| V_{D^*_j,n}
     -f^*\right\|_{L_2(\mfd)}^2
     \leq
     36^r c_5^2c_3^{-2r}\|f^*\|^2_{\mathbb W_2^r(\mfd)}|D|^{-\frac{2r}{2r+d}}.
\end{equation}
Using \eqref{bound first term.1} and \eqref{bound second term.1}, and Lemma \ref{Lemma:distributed.1},
\begin{equation*}
    \mathbf{E}\left\{\|V_{D,n}^{(m)}-f^*\|_{L_2(\mfd)}^2\right\} \leq C_2|D|^{-\frac{2r}{2r+d}}.
\end{equation*}
Here $C_2 = 2^{2r+1}\cdot 3^{2r} c_5^2 c_3^{-2r}\|f^*\|^2_{\mathbb W_2^r(\mfd)} + 3^{-d} c_1 c_2^2 c_3^{d} M^2$.
\end{proof}

We will use the following lemma to prove Theorem~\ref{thm:dfh_ran_noisy}, which can be obtained similarly as Lemma~\ref{Lemma:distributed.1}.
\begin{lemma}\label{Lemma:distributed}
For the distributed filtered hyperinterpolation $V_{D,n}^{(m)}$ with random sampling points,
\begin{align*}
    \mathbf{E}\left\{\|V_{D,n}^{(m)}-f^*\|_{L_2(\mfd)}^2\right\}
        &\leq
         \sum_{j=1}^m\frac{|D_j|^2}{|D|^2}
         \mathbf{E}\left\{\|V_{D_j,n}-f^*\|_{L_2(\mfd)}^2\right\}\\
        &\qquad+\sum_{j=1}^m\frac{|D_j|}{|D|}
        \left\|\mathbf{E}\{V_{D_j,n}\}-f^*\right\|_{L_2(\mfd)}^2.
\end{align*}
\end{lemma}

\begin{proof}[Theorem \ref{thm:dfh_ran_noisy}] 
By Lemma \ref{Lemma:distributed}, we only need to estimate the bounds of $\mathbf{E}\bigl\{\|V_{D_j,n}-f^*\|_{L_2(\mfd)}^2\bigr\}$ and $\left\|\mathbf{E}\{V_{D_j,n}\}-f^*\right\|_{L_2(\mfd)}^2$.
To estimate the first, we obtain from \eqref{IMportant} with $D=D_j$
that for $j=1,\dots,m$,
\begin{align*}
    \mathbf{E}\left\{\|V_{D_j,n}-f^*\|^2_{L_2(\mfd)}\right\}
    &\leq
     c_5^2n^{-2r}\|f\|^2_{\mathbb W_2^r(\mfd)} + \frac{2c_1^2M^2n^d}{|D_j|}\notag\\
    &\quad + 2\|f^*\|_{L_\infty(\mfd)}^2\left(\mu(\mfd)+ 2c_1^2 n^{d}+2\right)
     \exp\bigl\{-C|D_j|/n^d\bigr\}.
\end{align*}
Since $\min_{1\leq j\leq m}|D_j|\geq |D|^\frac{d+\tau}{2r+d}$, $n\asymp |D|^{\frac{1}{2r+d}}$, $2r>d$ and $0<\tau<2r$,
$$
     2\|f^*\|_{L_\infty(\mfd)}^2\left(\mu(\mfd)+ 2c_1^2 n^{d}+2\right)
     \exp\bigl\{-C|D_j|/n^d\bigr\}\leq
      \tilde{C}_7|D|^{-\frac{2r}{2r+d}},
$$
where $\tilde{C}_7$ is a constant depending only on $r,c_1,C,d$ and
$f^*$. Thus, there exists a constant $\tilde{C}_8$ independent of
$m,n,|D_1|,\dots,|D_m|$ and $|D|$ such that
\begin{align}\label{bound first term}
         &\sum_{j=1}^m\frac{|D_j|^2}{|D|^2}
         \mathbf{E}\left\{\|V_{D_j,n}-f^*\|_{L_2(\mfd)}^2\right\}  \nonumber\\
         &\qquad\leq
        \tilde{C}_8 \left(|D|^{-\frac{2r}{2r+d}}+
        \sum_{j=1}^m\frac{|D_j|^2}{|D|^2}\frac{|D|^{\frac{d}{2r+d}}}{|D_j|}\right)
        = \bigl(\tilde{C}_8+1\bigr)|D|^{-\frac{2r}{2r+d}}.
\end{align}
To bound $\left\|\mathbf{E}\{V_{D_j,n}\}-f^*\right\|_{L_2(\mfd)}^2$, we use \eqref{condition expectation} and Jensen's inequality to obtain
\begin{align}\label{eq:err_fDjn_f*}
     \left\|\mathbf{E}\{V_{D_j,n}\}-f^*\right\|_{L_2(\mfd)}^2
          &=
          \left\|\mathbf{E}\{\mathbf{E}\{V_{D_j,n}|\Lambda_{ D_j }\}-f^*\}\right\|_{L_2(\mfd)}^2\notag\\
         &=
        \left\|\mathbf{E}\{V_{D^*_j,n}-f^*\}\right\|_{L_2(\mfd)}^2
         \leq
       \mathbf{E}\left\{\|V_{D^*_j,n}-f^* \|_{L_2(\mfd)}^2\right\}.
\end{align}
We now use the similar proof as Theorem~\ref{thm:ndfh_noisy_ran} to prove the error bound of distributed filtered hyperinterpolation $V_{D,n}^{(m)}$. For each $j=1,\dots,m$, we let $\Omega_{D_j}$ be the event such that the sum of the quadrature weights $\sum_{i=1}w_{i,n,D_j}^2\leq 2/|D_{j}|$, and $\Omega_{D_j}^c$ the complement of $\Omega_{D_j}$. We write 
\begin{align*}
      \mathbf{E}\left\{\|V_{D^*_j,n}-f^*\|^2_{L_2(\mfd)}\right\}
      &=\mathbf{E}\left\{\|V_{D^*_j,n}-f^*\|^2_{L_2(\mfd)}|\Omega_{D_j}\right\}\mathbf P\{\Omega_{D_j}\}\\
      &\qquad+\mathbf{E}\left\{\|V_{D^*_j,n}-f^*\|^2_{L_2(\mfd)}|\Omega_{D_j}^c\right\}\mathbf P\{\Omega_{D_j}^c\},
\end{align*}
where
\begin{equation*}
     \mathbf{E}\left\{\|V_{D^*_j,n}-f^*\|^2_{L_2(\mfd)}|\Omega_{D_j}^c\right\}\mathbf P\{\Omega_{D_j}^c\}
      \leq 4\|f^*\|_{L_\infty(\mfd)}^2\exp\bigl\{-C|D_j|/n^d\bigr\}.
\end{equation*}
By \eqref{bound a} with $D=D_j$,
\begin{align*}
     &\mathbf{E}\left\{\|V_{D^*_j,n}-f^*\|^2_{L_2(\mfd)}
      |\Omega_{D_j}\right\}\mathbf P\{\Omega_{D_j}\} \\
     &\quad\leq
       c_5^2n^{-2r}\|f\|^2_{\mathbb W_2^r(\mfd)}+ 2\|f^*\|_{L_\infty(\mfd)}^2(\mu(\mfd)+ 2c_1^2 n^{d})\exp\bigl\{-C|D_j|/n^d\bigr\}.
\end{align*}
These give
\begin{align*}
    &\mathbf{E}\left\{\|V_{D^*_j,n}-f^*\|^2_{L_2(\mfd)}\right\}\\
    &\quad\leq c_5^2n^{-2r}\|f\|^2_{\mathbb W_2^r(\mfd)} + 2\|f^*\|_{L_\infty(\mfd)}^2(\mu(\mfd)
     +2c_1^2n^{d}+2)\exp\bigl\{-C|D_j|/n^d\bigr\}.
\end{align*}
By $\min_{1\leq j\leq m}|D_j|\geq |D|^\frac{d+\tau}{2r+d}$, $n\sim|D|^{\frac{1}{2r+d}}$ and $2r>d$, $0<\tau<2r$,
\begin{equation*}
     \mathbf{E}\left\{\|V_{D^*_j,n}-f^*\|^2_{L_2(\mfd)}\right\}
        \leq \tilde{C}_9|D|^{-\frac{2r}{2r+d}},
\end{equation*}
which with \eqref{eq:err_fDjn_f*} and $\sum_{j=1}^m\frac{|D_j|}{|D|}=1$ gives
\begin{equation}\label{bound second term}
     \sum_{j=1}^m\frac{|D_j|}{|D|}\left\|\mathbf{E}\{V_{D_j,n}\}-f^*\right\|_{L_2(\mfd)}^2
     \leq \tilde{C}_{9}|D|^{-\frac{2r}{2r+d}}.
\end{equation}
Using \eqref{bound first term} and \eqref{bound second term}, and
Lemma \ref{Lemma:distributed}, we then obtain
\begin{equation*}
\mathbf{E}\left\{\|V_{D,n}^{(m)}-f^*\|_{L_2(\mfd)}^2\right\}
     \leq
     C_4|D|^{-\frac{2r}{2r+d}},
\end{equation*}
thus completing the proof.
\end{proof}


\newpage

\section{Table of notations} 
\label{app:notations}
\thispagestyle{empty}

\scalebox{.9}{
\begin{tabular}{l|p{0.71\textwidth}}
\toprule
\textbf{Symbol} & \textbf{Meaning}\\
\toprule
$\N$ & Set of natural numbers $\{1,2,\dots\}$\\
$\Nz$ & $\N\cup \{0\}$\\
$\Rd[d]$ & $d$-dimensional real coordinate space\\
$\Rplus$ & Set of non-negative real numbers\\
$\mfd$     & Compact and smooth Riemannian manifold, where we call $\mfd$ $d$-manifold\\
$d$        & Dimension of manifold $\mathcal{M}$\\
$B(\bx,\alpha)$ & Ball with center $\bx$ and radius $\alpha$ in manifold \\
$\rho(\bx,\by)$ & Distance between points $\bx,\by\in \mfd$ induced by Riemannian metric\\
$\mu$ & Lebesgue measure on $\mfd$\\
$C(\mfd)$ & Continuous function space on $\mfd$\\
$L_p(\mfd)$ & Real-valued $L_p$ space on $\mfd$\\
$r$ & Smoothness index of Sobolev space containing the target function\\
$W^r_p(\mfd)$ & Sobolev space on $\mfd$ with smoothness $r$ and $p$-th norm\\
$n$  & Degree of polynomial or polynomial space on $\mfd$\\
$\Pi_n$ & Polynomial space of degree $n$ on $\mfd$\\
$P_n$ & Diffusion polynomial of degree $n$ on $\mfd$\\
$E_n(f)_p$ & Best approximation for $f$ in $L_p(\mfd)$\\
$\LBm$ & Laplace-Beltrami operator on $\mfd$\\
$\eigfm$ & The $\ell$th eigenfunction of Laplace-Beltrami operator on $\mfd$\\
$\eigvm$ & The $\ell$th eigenvalue of Laplace-Beltrami operator on $\mfd$\\ 
$D$ & Data set of $|D|$ pairs of sampling points $\bx_i$ and real data $y_i$\\
$D^*$ & Clean data set of pairs of sampling points $\bx_i$ and real data $f^*(\bx_i)$ for ideal function $f^*$\\
$|D|$ or $N$ & Number of elements of a data set $D$\\
$m$ & Number of servers in distributed filtered hyperinterpolation\\
$\{D_j\}_{j=1}^m$ & Set of $m$ distributed data sets for a data set\\
$\Lambda_D$ & Set of sampling points $\bx_i$ of a data set $D$\\
$\QH$ & Quadrature rule, a set of $N$ pairs of real weights and points on the manifold\\
$\QH^{(m)}$ & Quadrature rule for distributed filtered hyperinterpolation with $m$ servers and weights given by \eqref{eq:wi_dfh_ran}\\
$f^\ast$   & Ideal target function $\mathcal{M}\to \mathbb{R}$ (noiseless outputs)\\
$f$        & Noisy function ($f^\ast$ plus noise)\\
$\epsilon_i$ & Noise for $i$th sample\\
$H$ & Filter in Definition~\ref{def:fiH}\\
$V_n(f)$ or $V_{n,\fiH}(f)$ & Filtered approximation of degree $n$ for function $f$ with filter $H$\\
$V_{D,n}$ & Non-distributed filtered hyperinterpolation with degree $n$ 
for data $D$ in Definition~\ref{defn:nondfh_det}\\[1mm]
$V_{D,n}^{(m)}$ & Distributed filtered hyperinterpolation with degree $n$ and $m$ servers for data $D$ in Definition~\ref{defn:dfh}\\
$V_{D^*,n}$ & Non-distributed filtered hyperinterpolation with degree $n$ for noiseless data $D^*$ in \eqref{eq:VDnf}\\[1mm]
$V_{D^*,n}^{(m)}$ & Distributed filtered hyperinterpolation with degree $n$ and $m$ servers for noiseless data $D^*$ in \eqref{eq:dfh_clean}\\
\bottomrule
\end{tabular}
}
\end{document}